\documentclass[12pt,letterpaper, reqno]{amsart}
\usepackage[margin=2.5cm]{geometry}
\usepackage{amscd}
\usepackage{amssymb}
\usepackage{amsthm}
\usepackage{graphicx}
\usepackage{color}
\usepackage[all]{xy}
\usepackage{mathrsfs}
\usepackage{marvosym}
\usepackage{stmaryrd}
\usepackage{srcltx}
\usepackage{mathtools}
\usepackage{scalerel,amssymb}
\usepackage{tikz-cd}
\usepackage{adjustbox}
\usepackage{placeins}
\usepackage{floatrow}
\usepackage{setspace}

\usepackage{todonotes}
\newfloatcommand{capbtabbox}{table}[][\FBwidth]
\definecolor{amaranth}{rgb}{0.9, 0.17, 0.31}
\usepackage[colorlinks=true,citecolor=amaranth,linkcolor=black]{hyperref}%
\usetikzlibrary{shapes,arrows,chains}

\let\oldtocsection=\tocsection
\let\oldtocsubsection=\tocsubsection
\let\oldtocsubsubsection=\tocsubsubsection
\renewcommand{\tocsection}[2]{\hspace{0em}\oldtocsection{#1}{#2}}
\renewcommand{\tocsubsection}[2]{\hspace{1em}\oldtocsubsection{#1}{#2}}
\renewcommand{\tocsubsubsection}[2]{\hspace{0.5em}\oldtocsubsubsection{#1}{#2}}
\newtheorem{theorem}{Theorem}[section]
\newtheorem{lemma}[theorem]{Lemma}

\newtheorem{corollary}[theorem]{Corollary}
\newtheorem{conjecture}[theorem]{Conjecture}

\newtheoremstyle{defstyle}
  {.6em} % Space above
  {.1em} % Space below
  {} % Body font
  {} % Indent amount
  {\bfseries} % Theorem head font
  {.} % Punctuation after theorem head
  {.5em} % Space after theorem head
  {} % Theorem head spec (can be left empty, meaning `normal')
\theoremstyle{defstyle} \newtheorem{definition}[theorem]{Definition}

\newtheorem{example}[theorem]{Example}

\theoremstyle{remark}
\newtheorem{remark}[theorem]{Remark}

\numberwithin{equation}{section}
\numberwithin{figure}{section}

\newcommand{\CC} {\mathbb{C}}

\newcommand{\cO} {\mathcal{O}}

\newcommand{\PP} {\mathbb{P}}
\newcommand{\QQ} {\mathbb{Q}}
\newcommand{\RR} {\mathbb{R}}

\newcommand{\ZZ} {\mathbb{Z}}

\newcommand{\ev} {\mathrm{ev}}
\newcommand{\pt}{\mathrm{pt}}
\newcommand{\vir}{\mathrm{vir}}

\begin{document}

\title[BPS polynomials and Welschinger invariants]{BPS polynomials and Welschinger invariants}

\author[H.\,Arg\"uz]{H\"ulya Arg\"uz}
\address{University of Georgia, Department of Mathematics, Athens, GA 30605}
\email{Hulya.Arguz@uga.edu}

\author[P.\,Bousseau]{Pierrick Bousseau}
\address{University of Georgia, Department of Mathematics, Athens, GA 30605}
\email{Pierrick.Bousseau@uga.edu}

\begin{abstract}
We generalize Block-G\"ottsche polynomials, originally defined for toric del Pezzo surfaces, to arbitrary surfaces. To do this, we show that these polynomials arise as special cases of BPS polynomials, defined for any surface $S$ as Laurent polynomials in a formal variable $q$ encoding the BPS invariants of the $3$-fold $S \times \mathbb{P}^1$. We conjecture that for surfaces $S_n$ obtained by blowing up $\mathbb{P}^2$ at $n$ general points, the evaluation of BPS polynomials at $q=-1$ yields Welschinger invariants, given by signed counts of real rational curves. We prove this conjecture for all surfaces $S_n$ with $n \leq 6$. 
\end{abstract}
\newtheoremstyle{cited}%
  {3pt}% (space above)
  {3pt}% (space below)
  {\itshape}% (body font)
  {}% (indent amount)
  {\bfseries}% {theorem head font}
  {.}% {punctuation after theorem head}
  {.3em}% {space after theorem head}
  {\thmname{#1} \thmnumber{#2}\thmnote{\normalfont#3}}% {theorem head spec}

\theoremstyle{cited}
\newtheorem{citedthm}{Theorem}
\newtheorem{citedconj}[citedthm]{Conjecture}
\renewcommand*{\thecitedthm}{\Alph{citedthm}}

\maketitle

\setcounter{tocdepth}{2}
\tableofcontents
\setcounter{section}{0}

\section{Introduction}

\subsection{Overview}

BPS invariants are integers underlying the higher genus Gromov--Witten theory of $3$-folds. For any smooth projective surface $S$, we organize the BPS invariants of the $3$-fold $S \times \PP^1$ into BPS polynomials, which are Laurent polynomials in a formal variable $q$. These polynomials refine counts of complex rational curves in $S$, in the sense that they specialize at $q=1$ to genus zero Gromov--Witten invariants of $S$. In the case of toric del Pezzo surfaces, we show that BPS polynomials coincide with Block–G\"ottsche polynomials defined using tropical geometry. 

Beyond the toric setting, we conjecture that for surfaces $S_n$ obtained by blowing up $\mathbb{P}^2$ at $n$ general points, evaluating BPS polynomials at $q = -1$ recovers the Welschinger invariants, which are signed counts of real rational curves. Using Brugall\'e’s floor diagram techniques, we verify a relative version of this conjecture for all $n$. Furthermore, employing a refined version of the Abramovich–Bertram–Vakil formula for $n = 6$, we prove the main conjecture for all surfaces $S_n$ with $n \leq 6$. This establishes a striking interpolation between real and complex curve enumerations through higher genus Gromov--Witten theory. 

Finally, we conjecture that BPS polynomials of $S$ can be expressed in terms of K-theoretic refined BPS invariants of the non-compact Calabi--Yau 3-fold $K_S$, the total space of the canonical line bundle over $S$. This predicts a surprising relation between higher genus Gromov--Witten theory of $S\times \PP^1$ and refined sheaf counting theory of $K_S$.

\subsection{BPS polynomials}

Gromov–Witten theory of a complex $3$-fold $X$ yields a wealth of numerical invariants, owing to the fact that the virtual dimension of any curve-counting problem remains independent of the curve's genus. Specifically, for a fixed curve class $\beta \in H_2(X,\ZZ)$, and a set of cohomology insertions $\gamma$, we obtain Gromov–Witten invariants $GW_{g,\beta,\gamma}^X$, for every genus $g$. These invariants are virtual counts of genus $g$ complex curves in $X$ of class $\beta$ satisfying constraints imposed by $\gamma$. Although Gromov--Witten invariants are typically rational numbers, remarkably they can be encoded into integer values $BPS_{g,\beta,\gamma}^X$, known as BPS invariants, as first conjectured by Gopakumar--Vafa \cite{GV1, GV2, Pandh1999, Pandha3questions} and proved in \cite{IP, zinger2011}.

By \cite{doan2021gopakumar, DW23}, for fixed $\beta$ and $\gamma$, there exist only finitely many values of $g$ such that $BPS_{g,\beta,\gamma}^X \neq 0$. Hence, the BPS invariants $BPS_{g,\beta, \gamma}^X$ can be naturally arranged into a Laurent polynomial in a formal variable $q$, referred to as a BPS polynomial, given as in Definition \ref{def_bps_3fold} by
\[ BPS^X_{\beta,\gamma}(q) := \sum_{g \geq 0} BPS_{g,\beta, \gamma}^X
(-1)^g (q-2+q^{-1})^{g} \in \ZZ[q^{\pm}] \,. \] 
%We refer to $BPS^X_{\beta,\gamma}(q)$ as the BPS polynomial of $X$ of class $\beta$ with insertions $\gamma$.
In this paper, we study BPS polynomials associated with $3$-folds of the form $X = S \times \mathbb{P}^1$, where $S$ is a smooth projective surface. 
For every $\beta \in H_2(S,\ZZ)$ such that \[m_{\beta} := -1 +  c_1(S)\cdot\beta \geq 0\,,\] we define in \S \ref{section_bps_surfaces} the BPS polynomial of $S$ of class $\beta$ by
\[BPS_\beta^S(q):=BPS_{(\beta,0),\gamma_\beta}^{S \times \PP^1}(q)\,,\] where $BPS_{(\beta,0),\gamma_\beta}^{S \times \PP^1}(q)$  is the BPS polynomial of $S \times \PP^1$ of class $(\beta,0)\in H_2(S \times \PP^1,\ZZ)$, 
with insertions $\gamma_\beta$ lifted from $m_{\beta}$ points in $S$ and one point in $\PP^1$. 
As established in \S\ref{section_bps_surfaces}, the BPS polynomials $BPS_\beta^S(q)$ encode Gromov--Witten invariants $GW_{g,\beta}^S$ of the surface $S$ with $\lambda_g$-class insertion as defined in \eqref{eq_gw_S}. 
The $\lambda_g$-class insertion originates from the discrepancy between the obstruction theories for curves in $S$ and those in $S \times \PP^1$.
Explicitly as stated in \eqref{eq_bps_S}, under the substitution $q=e^{iu}$, we get
\begin{equation*} 
%\label{eq_bps_S}
BPS_\beta^S(q)=\left(2 \sin \left(\frac{u}{2}\right) \right)^{1-m_\beta} \sum_{g \geq 0}
GW_{g,\beta}^S u^{2g-1+m_\beta} \,.
\end{equation*}
Consequently, for $q=1$, the value $BPS_\beta^S(1)$ is equal to the Gromov--Witten count $GW_{0,\beta}^S$ of rational curves in $S$ of class $\beta$ passing through $m_{\beta}$ general points. Hence, one can view the polynomial $BPS_\beta^S(q)$ as a natural refinement of  $GW_{0,\beta}^S$.

For any $3$-fold $X$, the Gromov--Witten/pairs correspondence \cite{MNOP1, MNOP2}\cite[Conjecture 3.28]{PT1}, proved in many cases \cite{MOOP, PP, pardon2023universally}, predicts that the BPS invariants $BPS_{g,\beta,\gamma}^X$ of $X$ can be alternatively described in terms of Pandharipande--Thomas invariants $PT_{\beta,\chi,\gamma}^X$. These invariants are defined via moduli spaces of stable pairs $(\cO_X \rightarrow F)$ where $F$ is a one-dimensional sheaf satisfying $[F]=\beta$ and $\chi(F)=\chi$, and with insertions determined by $\gamma$ \cite{PT1, PT2}. 
When $X=S \times \PP^1$ and $m_\beta\geq 0$, we demonstrate in \S \ref{sec_stable_pairs} that the general Gromov--Witten/pairs correspondence simplifies to the following expression with the BPS polynomial of $S$:
\[ \sum_{\chi \in \ZZ} PT_{\beta,\chi,\gamma_\beta}^{S \times \PP^1} (-q)^{\chi}=-q (1-q)^{m_\beta-1} BPS_\beta^S(q)\,.\]

The main objective of this paper is to establish that BPS polynomials encode invariants in real algebraic geometry, capturing the signed enumeration of real curves in rational surfaces, known as Welschinger invariants.

\subsection{BPS polynomials and Welschinger invariants}
\subsubsection{Welschinger invariants}
Let $S$ be a rational smooth projective surface over $\CC$. Up to deformation, $S$ is either isomorphic to $\PP^1 \times \PP^1$ or to a surface $S_n$ obtained by blowing-up $n$ general points in $\PP^2$. Given a real structure, that is, an anti-holomorphic involution, on $S$, and a class
$\beta \in H_2(S,\ZZ)$ satisfying $m_{\beta} = -1 +  c_1(S)\cdot\beta \geq 0$, Welschinger introduced a signed count of real rational curves in $S$ of class $\beta$ passing through a general real configuration $x$ of $m_\beta$ points (see \S\ref{section_real} for the description of the Welschinger signs). Remarkably, these signed counts define invariants, in the sense that they depend only on the deformation class of the real structure and on the number of real points in $x$ \cite{We, We1}. These invariants, known as Welschinger invariants, play a fundamental role in real algebraic geometry, providing lower bounds for the enumeration of real curves. They have been extensively studied, particularly in the context of del Pezzo surfaces \cite{ABLM, BrugFloorconic, brugalle2021invariance, BP1, BP2, itenberg2003appendix, IKS_CH, IKS3, IKS0, IKS2}. Moreover, Welschinger invariants can be interpreted as examples of open Gromov–Witten invariants \cite{cho_open_gw, open_gw, solomon2006intersection, ST_open_gw}.

In this paper, we work with the standard real structure on $S$. When $S=\PP^1 \times \PP^1$, this corresponds to the real structure whose real locus is $\mathbb{R}\mathbb{P}^1 \times \mathbb{R}\mathbb{P}^1$. For $S = S_n$, the standard real structure is induced from the standard real structure on $\mathbb{P}^2$, whose real locus is $\mathbb{R}\mathbb{P}^2$, by blowing up $n$ general real points. Additionally, when fixing a real configuration of $m_{\beta}$ points, we consider purely real configurations, consisting exclusively of real points without pairs of complex conjugate points. We denote the corresponding Welschinger invariant by $W_{\beta}^S$.

\subsubsection{The toric case} When $S$ is a toric del Pezzo surface, that is, $S=\PP^1 \times \PP^1$ or $S=S_n$ with $n \leq 3$, Block–G\"ottsche introduced Laurent polynomials $BG_{\beta}^{S}(q)$ in a formal variable $q$, using purely combinatorial techniques from tropical geometry. These polynomials refine complex rational curve counts on $S$, in the sense that evaluating them at $q = 1$ recovers the genus zero Gromov-Witten invariants $GW_{0,\beta}^S$ of $S$. Moreover, they have the remarkable property that evaluating at $q = -1$ yields the Welschinger invariants $W_\beta^S$.
In our first main result, Theorem \ref{thm_absolute_relative}, we prove:
\begin{citedthm} \label{thm_A}
Let $S$ be a toric del Pezzo surface. Then, for every $\beta \in H_2(S,\ZZ)$
such that $\beta \cdot D_j \geq 0$ for every toric divisor $D_j$ of $S$, the BPS polynomial $BPS_\beta^{S}(q)$ is equal to the Block-G\"ottsche polynomial $BG_\beta^{S}(q)$:
\[ BPS_\beta^{S}(q) = BG_\beta^{S}(q) \,.\]
\end{citedthm}
As established in  \cite{Bou2019}, Block–G\"ottsche invariants can be expressed in terms of higher genus log Gromov–Witten invariants of $S$ relative to its toric boundary. Consequently, to prove Theorem \ref{thm_A}, it suffices to show a correspondence between Gromov--Witten invariants of $S$ and log Gromov--Witten invariants of $S$ relative to the toric boundary. This is achieved in Theorem \ref{thm_absolute_relative_0}, which is proved through a degeneration argument.

\subsubsection{The non-toric case}
For $n>3$, the surface $S_n$ is no longer toric, and so the Block-G\"ottsche polynomials are no longer defined.
Nevertheless, the BPS polynomials are still defined and
we expect that the relationship with Welschinger invariants at $q=-1$ persists in this broader non-toric context, as proposed in Conjecture \ref{conj_bps_w}:

\begin{citedconj}
\label{conj_intro}
For every $n \geq 0$, let $S_n$ be the blow-up of $\PP^2$ at $n$ general real points, and $\beta \in H_2(S_n,\ZZ)$ be a curve class such that $m_\beta:=-1+c_1(S_n)\cdot \beta \geq 0$.
Then, the specialization at $q=-1$ of the BPS polynomial $BPS_\beta^{S_n}$ coincides with the Welschinger count $W_\beta^{S_n}$ of real rational curves in $S_n$ passing through $m_\beta$ real points in general position:
    \[ BPS_{\beta}^{S_n}(-1)=W_{\beta}^{S_n} \,.\]
\end{citedconj}
We establish Conjecture \ref{conj_intro} for all $n \leq 6$ in Section \ref{sec_thm_main}. To do so, we first examine the blow-up $\widetilde{S}_n$ of $\mathbb{P}^2$ along $n$-points lying on a smooth conic $\mathfrak{C}$. Denoting the strict transform of $\mathfrak{C}$ by $\widetilde{\mathfrak{C}}$, we introduce relative BPS polynomials  $BPS_{g,\beta,(\mu,\nu)}^{\widetilde{S}_n/\widetilde{\mathfrak{C}}}(q)$ in Section \ref{sec: relative BPS}, employing higher genus relative Gromov--Witten theory for the pair $(\widetilde{S}_n, \widetilde{\mathfrak{C}})$. Here, $\nu = (\nu)_i$ for $1 \leq i \leq \ell(\nu)$ and $\mu = (\mu)_j$ for $1 \leq j \leq \ell(\mu)$ are ordered partitions as in \eqref{eq_mu_nu}. In a key result, Theorem \ref{thm: relative_BPS_W} we establish a relative version of our conjecture, for all $n \geq 0$:
\begin{citedthm}
\label{thm: relative bps welschinger intro}
For every $n \geq 0$, the specialization at $q=-1$ of the relative BPS polynomial $BPS_{\beta,(\mu,\nu)}^{\widetilde{S}_n/\widetilde{\mathfrak{C}}}(q)$ is equal, up to an explicit factor, to a Welschinger count 
$W_{\beta, (\mu,\nu),\mathbf{x}}^{\widetilde{S}_n/\widetilde{\mathfrak{C}}}$ of real rational curves passing through a configuration of real points $\mathbf{x}$, as defined in \S\ref{subsec:relative Welschinger counts}:
    \[ BPS_{\beta, (\mu,\nu)} ^{\widetilde{S}_n/\widetilde{\mathfrak{C}}}(-1)=
    \left( \prod_{j=1}^{\ell(\nu)} \frac{v_j}{[\nu_j]_\RR}\right)
    W_{\beta, (\mu,\nu), \mathbf{x}}^{\widetilde{S}_n/\widetilde{\mathfrak{C}}}\,, \]
where $[\nu_j]_\RR=1$ if $\nu_j$ is odd, and $[\nu_j]_\RR=2$ if $\nu_j$ is even.
\end{citedthm}
To prove this, we first express BPS polynomials in terms of refined counts of marked floor diagrams, as defined by Brugall\'e in \cite{BrugFloorconic}. 
These diagrams encode the combinatorial structure of curves lying in the central fiber of a degeneration of $\widetilde{S}_n$ to a union of $\PP^2$ with $m_{\beta}$-many copies of the Hirzebruch surface $\mathbb{F}_4$ and a blow-up of $\mathbb{F}_4$ at $n$-points -- see \S\ref{sec_floor} for details.
We then use that refined counts of marked floor diagrams at $q=-1$ coincide with Welschinger counts by  \cite[Theorem 3.12]{BrugFloorconic}.

In Theorem \ref{thm_refined_ABV}, we further establish a relative version of the Abramovich–Bertram–Vakil formula, which connects BPS polynomials to relative BPS polynomials in the case $n = 6$. By combining Theorem \ref{thm: relative bps welschinger intro} with this formula and its real version
\cite{BP1, BP2, IKS2}, we ultimately derive one of the main results of this paper in Theorem \ref{thm_main}:
\begin{citedthm}
\label{thm: main intro}
Conjecture \ref{conj_intro} holds for all $n \leq 6$.
\end{citedthm}
The proof of this result provides an effective algorithm for computing BPS polynomials for $n \leq 6$. For example, in the case $n = 6$ and $\beta = 2 c_1(S_6)$, in Example \ref{example_3240_refined} we determine the BPS polynomial
\[BPS_{\beta}^{S_6}(q) =q^{-4}+13q^{-3}+100 q^{-2}+547 q^{-1}+1918+547 q+100 q^2+13 q^3+q^4 \,,\]
which interpolates between the Gromov--Witten invariant $GW_{0,\beta}^{S_6}=3240$
at $q=1$ and the Welschinger invariant $W_\beta^{S_6}=1000$ at $q=-1$, previously computed in  \cite[\S 5.2]{GP1998} and \cite[Example 17]{IKS3} respectively. 

The methods used to prove Theorem \ref{thm: main intro} extend naturally to the cases $n = 7$ and $n = 8$. However, due to the increasing complexity of the associated floor diagram combinatorics \cite[\S6-\S7]{BrugFloorconic}, we choose not to include these cases in the present paper. Furthermore, although we formulate Conjecture \ref{conj_intro} for all $n \geq 0$, up to date an algebro-geometric description of Welschinger invariants is only available when $S_n$ is a del Pezzo surface, that is, for $n \leq 8$. As the description of Welschinger invariants requires a generic perturbation of the almost complex structure for $n\geq 9$ \cite{We, We1}, establishing the conjecture in this case would necessitate symplectic techniques.

%Finally, we also propose a conjectural link between BPS polynomials and K-theoretic refined BPS invariants.

\subsection{BPS polynomials and K-theoretic refined BPS invariants}

G\"ottsche–Shende  conjectured a connection between Block–G\"ottsche polynomials and the Hirzebruch genera of Hilbert schemes of points on universal curves \cite{GSrefined}. 
In Conjecture \ref{conj_DT}, we propose a version of this conjecture in the broader context of BPS polynomials.

Given a surface $S$, the total space $K_S$ of the canonical line bundle of $S$ is a non-compact Calabi--Yau 3-fold equiped with a $\CC^\star$-action scaling the fibers of the projection $K_S \rightarrow S$. The $K$-theoretic refined genus $0$ BPS invariant $\Omega_\beta^{K_S}(q) \in \ZZ[q^{\pm}]$ of $K_S$ with $m_\beta$ point insertions can be defined by $\CC^\star$-localization using moduli spaces of one-dimensional sheaves on $K_S$ with $m_\beta$ insertions pulled-back from points in $S$. Instead, one could use moduli spaces of stable pairs on $K_S$, and conjecturally extract the same invariants \cite{afgani2020refinements, NO, Thomas_K_VW, thomas2024refined}. 
\begin{citedconj}
    Let $S$ be a smooth projective surface, and $\beta \in H_2(S,\ZZ)$ such that $m_\beta:=-1+c_1(S)\cdot \beta \geq 0$. Then, the BPS polynomial $BPS_\beta^S(q)$ of $S$ coincides with the refined genus $0$ BPS invariant of $K_S$ with $m_\beta$ point insertions:
    \[ BPS_\beta^S(q) = \Omega_\beta^{K_S}(q)\,. \]
\end{citedconj}

This conjecture predicts a surprising relation between Gromov--Witten theory of $S \times \PP^1$, or equivalently the $\CC^\star$-equivariant Gromov--Witten theory of $S \times \CC$, and refined sheaf counting on $K_S$.
In \S\ref{section_refined_dt}, we argue that this conjecture should arise from the hypothetical existence of a 
$(\CC^\star)^2$-equivariant enumerative theory of the Calabi--Yau 5-fold $K_S \times \CC^2$ as in \cite{brini2024refined, NO}.
Additionally, we describe how known results about K3 surfaces and abelian surfaces provide evidence for natural extensions of this conjecture.
Finally in \S\ref{sec_K3_real}, we explore a connection between BPS polynomials and Welschinger invariants in the context of K3 surfaces.

\subsection{Related works}
Since Brugall\'e--Mikhalkin introduced floor diagrams for toric surfaces \cite{EM1_floor, EM2_floor}, these diagrams have been extensively used to enumerate curves in toric settings -- see for instance \cite{ABLM, bousseau_floor, CMR_floor, CJMR}. 
In this paper, we use refined counts of a version of floor diagrams relative to a conic as introduced by Brugalle in \cite{BrugFloorconic}, which concerns the non-toric geometry $(\widetilde{S}_n,\widetilde{\mathfrak{C}})$, to compute BPS polynomials.

In the toric situation, Mikhalkin also studied a particular class of Block--G\"ottsche polynomials and showed that they admit an interpretation in terms of counts of real curves weighted by a quantum index \cite{mikhalkin_quantum} -- see also \cite{MR4693576, itenberg2023real, itenberg2024real} for further generalizations. It is not obvious how such counts might be generalized beyond toric cases, and if they could be then related to the BPS polynomials defined in our paper in terms of higher genus Gromov--Witten invariants.

Brett Parker has provided in \cite{parker} an algorithm for computing the Gromov–Witten invariants $GW_{0,\beta}^{S_n}$ of $S_n$ for any $n \geq 0$. 
Building upon \cite{Bou2019, lo2022thesis}, a refined version of this algorithm can be obtained to determine the BPS polynomials of $S_n$ for all $n \geq 0$. 
This will be discussed further in future work.

\subsection{Acknowledgement}
We thank Erwan Brugall\'e for many helpful discussions and email exchanges related to \cite{BrugFloorconic, brugalle2021invariance} and Example \ref{example_3240_refined}. We also thank Ilia Itenberg for bringing to our attention works on real invariants of K$3$ surfaces, and Richard Thomas for very useful exchanges on his work on refined sheaf counting on local K$3$ surfaces. The last section of our paper was completed during the MIST workshop at the Chinese University of Hong Kong. We thank Conan Leung and other organizers for their hospitality. H\"ulya Arg\"uz is supported by the NSF grant DMS-2302116, and Pierrick Bousseau is supported by the NSF grant DMS-2302117.

%Refer Mikhalkin-Brugalle introducing Floor diagrams for h-transverse toric surfaces. Add more references on floor diagrams.
%Refer Vakil:counting curves on rational surfaces \cite{vakil2000counting}

%Talk of the refined Brett Parker algorithm ?

%Cite Maulik's student thesis: Chun Hong Lo \cite{lo2022thesis}

%Related/future work: more general real configurations of point constraints ($s \neq 0$), more general real structures on $S_n$.

%Thanks: talk at Oberwolfach 2024, Ilia Itenberg for remark on K3 story, Erwan Brugall\'e for discussion related to Example \ref{example_3240_refined}.

\section{Complex, real, and refined counts}
\label{section_complex_real_refined}

In this section, we begin with a brief overview of curve counting theories in surfaces over both complex and real numbers. 
We then describe the refined curve counts given by the Block--G\"ottsche polynomials defined using tropical geometry.

\subsection{Counting rational curves in complex surfaces}
\label{section_complex}
Let $S$ be a smooth projective surface over $\CC$. For every $m \in \ZZ_{\geq 0}$ and $\beta \in H_2(S,\ZZ)$, the moduli space $\overline{M}_{0,m}(S,\beta)$ of $m$-pointed genus zero stable maps to $S$ of class $\beta$ is a proper Deligne--Mumford stack. 
It carries a virtual fundamental class $[\overline{M}_{0,m}(S,\beta)]^{\vir}$ of dimension $-1+c_1(S) \cdot \beta+m$ \cite{berhend1997gw, fulton1996notes}.

From this point onward, we assume that $m_\beta:= -1+c_1(S) \cdot \beta$ is nonnegative. 
The \emph{Gromov--Witten invariant} $GW_{0,\beta}^S \in \QQ$ is defined by imposing $m_\beta$ point constraints on the rational curves of class $\beta$, that is,
\begin{equation}  \label{eq_gw}
GW_{0,\beta}^S \coloneqq \int_{[\overline{M}_{0,m_\beta}(S,\beta)]^{\vir}} \prod_{i=1}^{m_\beta}  \ev_i^* (\pt_S) \,, \end{equation}
where $\ev_i : \overline{M}_{0,m_\beta}(S,\beta) \to S$ is the evaluation map at the $i$-the marked point, and $\pt_S \in H^4(S,\ZZ)$ is the Poincar\'e dual class of a point in $S$.

For every $n \in \ZZ_{\geq 0}$, we denote by $S_n$ a smooth projective surface obtained by blowing up $n$ general points in the complex projective plane $\PP^2$. 
By deformation invariance in Gromov--Witten theory, the Gromov--Witten invariants $GW_{0,\beta}^{S_n}$ do not depend on the specific configuration of blown-up points. We have $H_2(S_n,\ZZ)=\ZZ H \oplus \bigoplus_{i=1}^n \ZZ E_i$,
where $H$ is the pullback of the class of a line in $\PP^2$, and $E_i$ are the classes of the exceptional curves of the blow-up $S_n \rightarrow \PP^2$.
Since $c_1(S_n)=3H-\sum_{i=1}^n E_i$, for every $\beta=d H -\sum_{i=1}^n a_i E_i$, we obtain
\begin{equation} \label{eq_m_beta}
m_\beta = 3d -\sum_{i=1}^n a_i -1\,.\end{equation}

According to \cite[\S 4]{GP1998}, there exist only finitely many genus zero stable maps to $S_n$ of class $\beta$ that pass through $m_\beta$ points in general position, and $GW_{0,\beta}^{S_n}$ is a count of these stable maps with integer multiplicities. In particular, the Gromov--Witten invariants $GW_{0,\beta}^{S_n}$ are nonnegative integers. By \cite[\S 3]{GP1998}, they can be recursively computed using the WDVV equation, that is, the associativity of the quantum product.
When $S_n$ is a del Pezzo surface, that is, for $n \leq 8$, it is proven in \cite[\S 4]{GP1998} that the Gromov--Witten invariants $GW_{0,\beta}^{S_n}$ are enumerative: every genus zero stable map to $S_n$ of class $\beta$ passing through $m_\beta$ points in general position is an immersion
$\PP^1 \rightarrow S_n$ with at worst nodal image, and $GW_{0,\beta}^{S_n}$ is equal to the number of these maps all counted with multiplicity one.

\begin{example} \label{ex_cubic_complex}
    Let $n=0$, so that $S_n=\PP^2$, and $\beta=3H$. Then, $GW_{0,\beta}^{S_0}$ is the number of rational cubic curves in $\PP^2$ passing through $m_\beta=8$ general points. It is well known that $GW_{0,\beta}^{S_0}=12$.
\end{example}

\begin{example} \label{ex_3240_complex}
    Let $n=6$, so that $S_6$ is a smooth cubic surface in $\PP^3$, and $\beta=2c_1(S_6)=6H-2\sum_{i=1}^6 E_i$. Then, we have $m_\beta = 11$ by \eqref{eq_m_beta}, and $GW_{0,\beta}^{S_6}=3240$ by \cite[\S 5.2]{GP1998}.
\end{example}

\subsection{Counting rational curves in real surfaces}
\label{section_real}

As in \S\ref{section_complex}, let $S_n$ be a smooth projective surface over $\CC$ obtained by blowing up $n$ general points in the complex projective plane $\PP^2$. 
Let $\iota$ be a \emph{real structure} on $S_n$, that is, an anti-holomorphic involution $\iota:S_n \rightarrow S_n$. We review the definition given by Welschinger \cite{We,We1} of signed counts of real curves in the real surface $(S_n, \iota)$.

Let $\beta \in H_2(S_n,\beta)$ be a curve class such that $m_\beta:= -1+c_1(S_n) \cdot \beta \geq 0$.
Let $x=(x_i)_{1 \leq i \leq m_\beta}$ be a real configuration of $m_\beta$ general points in $S_n$, that is a subset of $S_n$ invariant by $\iota$. 
The set $x$ is a union of $r$ real points, that is fixed by $\iota$, and of $s$ pairs of distinct complex conjugated points, that is permuted by $\iota$, so that $r+2s=m_\beta$.
Let $\omega$ be a standard symplectic structure on $S_n$, 
induced by the presentation of $S_n$ as a symplectic blow-up of $\PP^2$ equipped with the Fubini-Study symplectic form.
 An almost complex structure $J$ on $S_n$ is called $\iota$-compatible if $\iota$ is $J$-anti-holomorphic. When $J$ is $\iota$-compatible, the involution $\iota$ defines an action of $\ZZ/2\ZZ$ on the set of $J$-pseudo-holomorphic curves $C \rightarrow S_n$. Fixed points of this action are referred to as \emph{real $J$-pseudo-holomorphic curves}.

 According to \cite[Theorem 1.11]{We1}, if $J$ is a 
 sufficiently generic $\omega$-tame $\iota$-compatible almost complex structure, the set 
 $M_{J,x}$ of genus zero $J$-pseudo-holomorphic curve $f: C \rightarrow S_n$ of class $\beta$ passing through $x$ is finite. Moreover, for every $(f: C \rightarrow S_n) \in M_{J,x}$, we have $C \simeq \PP^1$, the map $f$ is an immersion, and the image $f(C)$ has at worst nodal singularities.
 By the comparison between algebraic and symplectic Gromov--Witten invariants \cite{li_tian_gw, siebert_gw}, the cardinality of $M_{J,x}$ coincides with the Gromov--Witten invariant $GW_{0,\beta}^{S_n}$ defined in \S \ref{section_complex}. However, the cardinality of the set $M_{J,x}(\RR)$ of real $J$-pseudo-holomorphic curves depends in general on $J$ and on $x$. Welschinger addressed this issue by introducing the following signed count.

For every $(f:  C \rightarrow S_n) \in M_{J,x}(\RR)$, a real node of $f(C)$ is either isolated, locally isomorphic to $x^2+y^2=0$, or the intersection of two transverse real branches, locally isomorphic to $x^2-y^2=0$. The \emph{Welschinger sign} $w(f) \in \{\pm 1\}$ of $f$ is defined by 
\[ w(f)=(-1)^{m(f)}\,,\]
where $m(f)$ denotes the number of isolated real nodes in $f(C)$.
By \cite[Theorem 1]{We1}, the signed count 
\[ W_{\beta, (r,s)}^{(S_n,\iota)}:= \sum_{f \in M_{J,x}(\RR)} w(f)\,,\]
of real $J$-pseudo-holomorphic curves is independent of $J$, and only depends on the real configuration $x$ via the numbers $(r,s)$ of real points and pairs of complex conjugated points in $x$ respectively. This signed count is also independent of the choice of the standard symplectic form $\omega$, as the space of standard symplectic forms is connected. We refer to  $W_{\beta, (r,s)}^{(S_n,\iota)}$ as the \emph{Welschinger invariants} of the real surface $(S_n, \iota)$.
By \cite{horev2012open} -- see also \cite[\S 5]{chen_zinger_wdvv}, they can be recursively computed using the open WDVV equation proved in \cite{chen_open_wdvv}. 
When $S_n$ is a del Pezzo surface, that is, for $n \leq 8$, then, as reviewed in \S \ref{section_complex}, the standard complex structure on $S_n$ can be chosen as $J$. Consequently, the Welschinger invariants can be defined algebro-geometrically as signed counts of real stable maps in this case -- see \cite{W_revisited}. 

Throughout this paper, we assume that $S_n$ is obtained from $\PP^2$ by blowing up $n$ general real points, and that $\iota$ is the corresponding \emph{standard real structure}.
In this case, the real locus $S_n(\RR)$ is diffeomorphic to the connected sum of $n+1$ copies of $\RR\PP^2$, that is, the unique compact non-orientable topological surface of Euler characteristic $1-n$. 
We also focus on the case of \emph{purely real} configurations of points, that is, with $r=m_\beta$ and $s=0$. To simplify the notation, we denote by 
\begin{equation}\label{eq_W}
W_\beta^{S_n}\end{equation}
the corresponding Welschinger invariants.

\begin{example} \label{ex_cubic_real}
    Let $n=0$, so that $S_n=\PP^2$, and $\beta=3H$. Then, we have $W_\beta^{S_0}=8$.
\end{example}

\begin{example} \label{ex_3240_real}
    Let $n=6$, so that $S_6$ is a smooth cubic surface in $\PP^3$, and $\beta=2c_1(S_6)=6H-2\sum_{i=1}^6 E_i$. Then, we have $W_\beta^{S_6}=1000$ by \cite[Example 17]{IKS3} -- see also \cite[\S 9]{horev2012open}.
\end{example}

\subsection{Refined tropical curve counting in toric surfaces}
\label{section_BG}

Let $S$ be a projective toric surface. Let $\rho_1, \dots, \rho_\ell$ be the rays of the fan of $S$ in $\RR^2$, with integral primitive directions $m_1, \dots, m_\ell \in \ZZ^2$, corresponding to the toric divisors $D_1, \dots, D_\ell$ in $S$. 
Consider a non-zero curve class 
$
\beta \in H_2(S,\ZZ)$ 
such that $\beta \cdot D_j\geq 0$ for all $1\leq j \leq n$.
By standard toric geometry, the balancing condition $\sum_{j=1}^\ell (\beta \cdot D_j) m_j=0$ is satisfied. 
In particular, there exist at least two toric divisors $D_j$ such that $\beta \cdot D_j \geq 1$, so we have $c_1(S) \cdot \beta \geq 2$, and we obtain
\begin{equation} \label{eq_positive_m}
     m_\beta:= -1+c_1(S) \cdot \beta\geq 1 \,.
\end{equation}

Given a set $\mathbf{P}$ of $m_\beta$ points in $\RR^2$,
consider the set $T_{\beta,\mathbf{P}}$ of parametrized rational tropical curves $h: \Gamma \rightarrow \RR^2$, with $\beta \cdot D_i$ unbounded edges of direction $m_i$ and multiplicity $1$ for all $1 \leq i \leq \ell$, and passing through $\mathbf{P}$. 
By \cite[Proposition 4.13]{Mikhalkin2005}, for general enough $\mathbf{P}$, the set $T_{\beta,\mathbf{P}}$ is finite. 
Moreover, for every $h: \Gamma \rightarrow \RR^2$ in $T_{\beta, \mathbf{P}}$, the domain graph $\Gamma$ is 3-valent. 
By \cite[Definition 2.16]{Mikhalkin2005}, the multiplicity of a vertex $v$ of $\Gamma$ is defined as
\[ m_v := |\det(w_{e_1}u_{e_1}, w_{e_2} u_{e_2})|\,,\]
where $e_1$, $e_2$, $e_3$ are the edges of $\Gamma$ incident to $v$, with weights $w_{e_1}$, $w_{e_2}$, $w_{e_3}$, and primitive integral direction vectors $u_{e_1}, u_{e_2}, u_{e_3}$ pointing outwards $v$. 
By the balancing condition, we have $\sum_{i=1}^3 w_{e_i} u_{e_i}=0$, and so $m_v$ is well-defined, independently of the choice of $e_1$ and $e_2$ among $e_1, e_2, e_3$.
Following 
\cite[Definition 3.5]{BG}, the refined multiplicity of $v$ is defined as the $q$-integer version of $m_v$:
\begin{equation} \label{eq_q_integer} 
[m_v]_q : = \frac{q^{\frac{m_v}{2}}-q^{-\frac{m_v}{2}}}{q^{\frac{1}{2}} - q^{-\frac{1}{2}}} = q^{-\frac{m_v-1}{2}}
\sum_{j=0}^{m_v-1} q^j\in \ZZ_{\geq 0}[q^{\pm \frac{1}{2}
}] \end{equation}
and the Block-G\"ottsche refined multiplicity of $h: \Gamma \rightarrow \RR^2$ as 
\[ BG_h(q):= \prod_{v} [m_v]_q\,,\]
where the product is over the vertices of $\Gamma$.
The Block-G\"ottsche polynomial is defined as 
\[ BG_{\beta, \mathbf{P}}^S(q) = \sum_{h \in T_{\beta, \mathbf{P}}} BG_h(q) \in \ZZ_{\geq 0}[q^{\pm \frac{1}{2}}] \,.\]
By \cite[Theorem 1]{MR3142257}, $BG_{\beta, \mathbf{P}}(q)$ is tropically deformation invariant, that is, does not depend on the particular general configuration  $\mathbf{P}$ of points in $\RR^2$. 
Consequently, we denote the Block-G\"ottsche polynomial simply as $BG_{\beta}^S(q)$ in what follows.

By the complex tropical correspondence theorem  \cite[Theorem 1]{Mikhalkin2005} -- see also \cite{NS}, the value of the Block-G\"ottsche polynomial at $q=1$ is the number of genus zero stable maps to $S$, of class $\beta$, that pass through a general configuration of $m_\beta$ points and do not contain any torus fixed point of $S$.
Similarly, by the real tropical correspondence theorem \cite[Theorem 6]{Mikhalkin2005}, the value of the Block-G\"ottsche polynomial at $q=-1$ equals the number of real genus zero stable maps to $S$, of class $\beta$, that pass through a general configuration of $m_\beta$ real points and do not contain any torus fixed point of $S$. Thus, the Block-G\"ottsche polynomials interpolate between counts of real and complex curves.

As in \S\ref{section_complex}, let $S_n$ be the blow-up of $\PP^2$ at $n$ general points. If $n \leq 3$, then $S_n$ is a toric surface, and so the Block-G\"ottsche polynomials $BG_{\beta}^S(q)$ are defined as above. Moreover, it follows from the enumerativity result of \cite[\S 3]{GP1998} reviewed in \S\ref{section_complex} that all the stable maps contributing to the Gromov--Witten invariants $GW_{0,\beta}^{S_n}$ and Welschinger invariants $W_{\beta}^{S_n}$ avoid the torus fixed points of $S_n$. Consequently, we obtain 
\begin{equation} \label{eq_BG_interpolation}
BG_\beta^{S_n}(1)=GW_{0,\beta}^{S_n} \,\,\, \text{and} \,\,\, BG_\beta^{S_n}(-1)=W_{\beta}^{S_n}\,. \end{equation}
Therefore, the Block-G\"ottsche polynomials remarkably interpolate between the Gromov--Witten and Welschinger invariants reviewed in \S \ref{section_complex} and \S \ref{section_real} respectively.

\begin{example}
Let $n=0$, so that $S_n=\PP^2$, and $\beta=3H$. Then, we have 
\[BG_\beta^{S_0}(q)=q^{-1}+10+q\,,\] 
interpolating between $GW_{0,\beta}^{S_0}=12$ at $q=1$ (see Example \ref{ex_cubic_complex} ) and $W_\beta^{S_0}=8$ at $q=-1$ (see Example \ref{ex_cubic_real}).
\end{example}

\begin{example}
    Let $n=6$, so that $S_6$ is a smooth cubic surface in $\PP^3$, and $\beta=2c_1(S_6)=6H-2\sum_{i=1}^6 E_i$. Since $n=6>3$, the surface $S_6$ is not toric, and thus the corresponding Block-G\"ottsche polynomial is not defined. 
    In the following section \S\ref{section_bps}, 
    we define BPS polynomials, which generalize Block-G\"ottsche polynomials to arbitrary values of $n$. The BPS polynomial for this example, interpolating between $GW_{0,\beta}^{S_6}=3240$ as in Example 
    \ref{ex_3240_complex} and $W_\beta^{S_6}=1000$ as in Example \ref{ex_3240_real}, is calculated in Example \ref{example_3240_refined}.
\end{example}

\section{BPS polynomials}
\label{section_bps}

In \S\ref{section_BPS_polynomials_3folds}-\ref{section_bps_surfaces}, we introduce BPS polynomials of a surface $S$ using higher genus Gromov--Witten theory of the 3-fold $S \times \PP^1$.
Building on the main result of \cite{Bou2019}, which provides an interpretation of Block-G\"ottsche polynomials in terms of higher genus log Gromov--Witten invariants, we prove in \S\ref{section_BPS_BG} that the BPS polynomials recover Block-G\"ottsche polynomial when $S$ is a toric del Pezzo surface. 
Finally, in \S \ref{section_conjecture}, we conjecture that for any $n$, the specialization at $q=-1$ of the BPS polynomials of the surface $S_n$ -- the blow-up of $\PP^2$ at $n$ general points -- coincides with the Welschinger invariants $W_\beta^{S_n}$. Afterwards in 
\S \ref{section_bps_welschinger}, we prove this conjecture for $n \leq 6$.

\subsection{BPS polynomials of 3-folds}
\label{section_BPS_polynomials_3folds}
While the primary focus of this paper is the enumerative geometry of a surface $S$, the definition of BPS polynomials given in \S\ref{section_bps_surfaces} is formulated in terms of the enumerative geometry of the 3-fold $S \times \PP^1$. We begin by reviewing key aspects of curve counting in $3$-folds, then introduce the concept of BPS polynomials within this broader context.

\subsubsection{BPS polynomials and Gromov--Witten invariants}
Let $X$ be a smooth projective 3-fold over $\CC$. For every $g \in \ZZ_{\geq 0}$, curve class $\beta \in H_2(X.\ZZ)$, and cohomology classes $\gamma = (\gamma_1, \dots, \gamma_k)$ with $\gamma_i \in H^\star(X,\ZZ)$, the genus $g$ \emph{Gromov--Witten invariant} of $X$, of class $\beta$ and with insertion of $\gamma$, is defined as
\begin{equation} \label{eq_gw_X}
GW_{g,\beta,\gamma}^X
:= \int_{[\overline{M}_{g,k}(X,\beta)]^{\vir}}\prod_{i=1}^k \mathrm{ev}_i^\star(\gamma_i) \in \QQ\,, 
\end{equation}
where $\overline{M}_{g,k}(X,\beta)$ is the moduli space of $k$-pointed genus $g$ stable maps to $X$ of class $\beta$, $\mathrm{ev}_i: \overline{M}_{g,k}(X,\beta) \rightarrow X$ is the evaluation map at the $i$-th marked point, and $[\overline{M}_{g,k}(X,\beta)]^{\vir}$ is the virtual fundamental class, which is of dimension $c_1(X) \cdot \beta +k$.

From now on, we assume that $\beta$ is a \emph{Fano class}, meaning that $c_1(X) \cdot \beta >0$. In this case, following \cite[\S 0.4]{Pandh1999}\cite[\S 3]{Pandha3questions}, 
the \emph{Gopakumar--Vafa BPS invariants} $BPS_{g,\beta,\gamma}\in \QQ$ are defined by the formula
\begin{equation}
\label{Eq: gw_bps}
\sum_{g \geq 0} GW_{g,\beta,\gamma}^X u^{2g-2+c_1(X) \cdot \beta} = \sum_{g\geq 0} BPS_{g,\beta, \gamma}^X
\left( 2\sin \left( \frac{u}{2} \right) \right)^{2g-2+c_1(X)\cdot \beta} \,.\end{equation}
It is shown in \cite[Theorem 1.5]{zinger2011} by symplectic methods that the BPS invariants are integers: $BPS_{g,\beta,\gamma}^X \in \ZZ$.
Furthermore, by \cite[Corollary 1.16]{DW23}, for fixed $\beta$ and $\gamma$, we have $BPS_{g,\beta, \gamma}^X=0$ for large enough $g$. 
Therefore, we can define a Laurent polynomial in the variable $q=e^{iu}$ as follows.
\begin{definition}
\label{def_bps_3fold}
Let $X$ be a smooth projective $3$-fold, $\beta \in H_2(X, \ZZ)$ a Fano class, and $\gamma = (\gamma_1, \dots, \gamma_k)$ with $\gamma_i \in H^\star(X,\ZZ)$. Then, the \emph{BPS polynomial} of $X$ of class $\beta$ and with insertion of $\gamma$ is the Laurent polynomial defined by
\begin{align} 
\label{Eq:BPS}
BPS_{\beta, \gamma}^X(q)
:=& \sum_{g \geq 0} BPS_{g,\beta,\gamma}^X
\left( 2 \sin \left( \frac{u}{2} \right) \right)^{2g}
=\sum_{g \geq 0} BPS_{g,\beta, \gamma}
(-1)^g (q^{\frac{1}{2}} - q^{-\frac{1}{2}})^{2g} \\
\nonumber
=&\sum_{g \geq 0} BPS_{g,\beta, \gamma}^X
(-1)^g (q-2+q^{-1})^{g} \in \ZZ[q^{\pm}]\,.
\end{align}
\end{definition}

By Definition \ref{def_bps_3fold}, the BPS polynomials $BPS_{\beta,\gamma}^X(q)$ are Laurent polynomials in $q$ which are symmetric under $q \mapsto q^{-1}$.
An equivalent description of the BPS polynomials is obtained by rewriting \eqref{Eq:BPS} using \eqref{Eq: gw_bps}: 
\begin{equation}
\label{Eq:alternativeBPS}
BPS_{\beta,\gamma}^X(q)=
\left(2 \sin \left( \frac{u}{2}\right) \right)^{2-c_1(X)\cdot \beta}\sum_{g \geq 0} GW_{g,\beta, \gamma}^X u^{2g-2+c_1(X) \cdot \beta}\,.
\end{equation}

Although we initally defined  BPS polynomials in terms of Gromov--Witten invariants, in what follows we note that they can also be described using unramified Gromov--Witten invariants.

\subsubsection{BPS polynomials and unramified Gromov--Witten invariants}

Let $X$ be a smooth projective $3$-fold, $\beta \in H_2(X, \ZZ)$ a Fano class, and $\gamma = (\gamma_1, \dots, \gamma_k)$ with $\gamma_i \in H^\star(X,\ZZ)$. By the main result of \cite{nesterov2024unramified}, the BPS invariants $BPS_{g,\beta,\gamma}^X$ are equal to the \emph{unramified Gromov--Witten invariants} of $X$, defined by Kim--Kresh--Oh \cite{MR3118390}:
\begin{equation} \label{eq_unramified}
BPS_{g,\beta,\gamma}^X =  \int_{[\overline{M}^{\mathrm{un}}_{g,k}(X,\beta)]^{\vir}}\prod_{i=1}^k \mathrm{ev}_i^\star(\gamma_i)\,,\end{equation}
where $[\overline{M}^{\mathrm{un}}_{g,k}(X,\beta)$ is a moduli space of unramified stable maps to iterated blow-ups of points of $X$ -- see \cite{MR3118390, nesterov2024unramified} for details and \cite[\S 5 1/2]{MR3221298} for a brief exposition. The description of BPS invariants by unramified Gromov--Witten invariants has the advantage of avoiding the change of variables $q=e^{iu}$. However, it is currently not known how to prove the integrality $BPS_{g,\beta,\gamma}^X \in \ZZ$ and the vanishing $BPS_{g,\beta,\gamma}^X=0$ for large enough $g$ directly from the definition via unramified Gromov--Witten invariants -- see \cite[\S 1.3]{nesterov2024unramified}.

\subsection{BPS polynomials of surfaces}
\label{section_bps_surfaces}
\subsubsection{BPS polynomials and surface Gromov--Witten invariants}
Let $S$ be a smooth projective surface over $\CC$. In this section, we define BPS polynomials of $S$ as particular BPS polynomials of the 3-fold $X=S \times \PP^1$ defined in Definition \ref{def_bps_3fold}.

To do this, we consider the natural projections:
\begin{equation}
\label{eq:projections}
\tikzstyle{line} = [draw, -latex']
\begin{tikzpicture}
    \node[] at (6,0) (Step 1) { $X=S\times \PP^1$};
      \node[] at (4,-1.6) (Step 2) {$S$};
       \node[] at (8,-1.6) (Step 3) {$\PP^1$};
   \path [line] (Step 1) -- node [text width=0.5cm,midway,above ] {$\pi_S$} (Step 2);
 \path [line] (Step 1) -- node [text width=0.2cm,midway,above ] {$\pi_{\PP^1}$} (Step 3);
\end{tikzpicture}    
\end{equation}
and for every $\beta \in H_2(S,\ZZ)$ such that $m_\beta:=-1+c_1(S) \cdot \beta \geq 0$, we define a $(m_\beta+1)$-tuple $\gamma_\beta=(\gamma_{\beta, i})_{0 \leq i \leq m_\beta}$ of cohomology classes $\gamma_{\beta,i} \in H^\star(X,\ZZ)$ by
\[ \gamma_{\beta,0}:= \pi_{\PP^1}^*(\pt_{\PP^1}) \in H^2(X,\ZZ) \,\,\, \text{and}\,\,\,
\gamma_{\beta,i}:= \pi_S^*(\pt_S) \in H^4(X,\ZZ)\,\, \, \text{for all}\,\, 1 \leq i \leq m_\beta \,,\]
where $\pt_{\PP^1} \in H^2(\PP^1,\ZZ)$ is the Poincar\'e dual class of a point in $\PP^1$, and 
$\pt_{\PP^1} \in H^2(S,\ZZ)$ is the Poincar\'e dual class of a point in $S$.
Moreover, the curve class $(\beta,0)\in H_2(X,\ZZ)=H_2(S,\ZZ) \times \ZZ$ is a Fano class on $X$ since we have $c_1(X) \cdot (\beta,0)=c_1(S) \cdot \beta \geq 1$ by the assumption $m_\beta \geq 0$ on $\beta$. Therefore, Definition \ref{def_bps_3fold} applies to define BPS polynomials of $X$ of class $(\beta,0)$.

\begin{definition}
\label{Defn:BPSS}
Let $S$ be a smooth projective surface over $\CC$ and $\beta \in H_2(S,\ZZ)$ such that 
$m_\beta:=-1+c_1(S)\cdot \beta \geq 0$.
The \emph{BPS polynomial} of $S$ of class $\beta$ is the BPS polynomial of the 3-fold $X=S \times \PP^1$ of class $(\beta,0)\in H_2(X,\ZZ)=H_2(S,\ZZ) \times \ZZ$ and with insertion of the cohomology classes $\gamma_\beta$:
\[  BPS_{\beta}^S(q) := BPS_{(\beta,0),\gamma_\beta}^X(q) \in \ZZ[q^\pm] \,. \]
\end{definition}

The following result shows that the BPS polynomials $BPS_{g,\beta}^S$ can be described in terms of Gromov--Witten invariants of $S$ with insertion of a top lambda class.
For any smooth projective surface $S$ over $\CC$, $g \in \ZZ_{\geq 0}$, and $\beta \in H_2(S,\ZZ)$ such that $m_\beta:=-1+c_1(S) \cdot \beta \geq 0$, the moduli space $\overline{M}_{g,m_\beta}(S,\beta)$ of $m_\beta$-marked genus $g$ stable maps to $S$ of class $\beta$ carries a virtual fundamental class  
$[\overline{M}_{g,m_\beta}(S,\beta)]^{\vir}$
of dimension $g-1+c_1(S) \cdot \beta +m_\beta= g +2m_\beta$.
We define the Gromov--Witten invariant $GW_{g,\beta}^S$ by imposing $m_\beta$ point constraints and inserting the top lambda class $(-1)^g\lambda_g$, that is, 
\begin{equation} \label{eq_gw_S}
GW_{g,\beta}^S :=  \int_{[\overline{M}_{g,m_\beta}(S,\beta)]^{\vir}} (-1)^g \lambda_g \prod_{i=1}^{m_\beta} \mathrm{ev}_i^\star(\pt_S) \in \QQ\,,\end{equation}
where $\ev_i : \overline{M}_{g,m_\beta}(S,\beta) \to S$ is the evaluation map at the $i$-the marked point, and $\pt_S \in H^4(S,\ZZ)$ is the Poincar\'e dual class of a point in $S$. 
Moreover, $\lambda_g=c_g(\mathbb{E}) \in H^{2g}(\overline{M}_{g,m_\beta}(S,\beta), \QQ)$ is the top Chern class of the Hodge bundle $\mathbb{E}$, that is, the rank $g$ vector bundle over $\overline{M}_{g,m_\beta}(S,\beta)$ with fiber $H^0(C,\omega_C)$ over a stable map $C \rightarrow S$ \cite[\S 4]{mumford}. 
Note that, for $g=0$, we have $(-1)^0 \lambda_0=1$, and so the Gromov--Witten invariants $GW_{g,\beta}^S$ indeed specialize to the genus zero Gromov--Witten invariants $GW_{0,\beta}^S$ introduced in \S \ref{section_complex}, as the notation suggests.

\begin{lemma} \label{lem_X_S}
Let $S$ be a smooth projective surface over $\CC$ and $\beta \in H_2(S,\ZZ)$ such that 
$m_\beta:=-1+c_1(S)\cdot \beta \geq 0$. Then,
for every $g\geq 0$, the Gromov--Witten invariant $GW^X_{g,(\beta,0),\gamma_\beta}$ of the 3-fold $X=S \times \PP^1$ defined in \eqref{eq_gw_X} and the Gromov--Witten invariant
$GW_{g,\beta}^S$ of $S$ defined in \eqref{eq_gw_S} are equal:
\begin{equation} \label{eq_X_S}
GW^X_{g,(\beta,0),\gamma_\beta}=GW_{g,\beta}^S \,.
\end{equation}
Moreover, with the change of variables $q=e^{iu}$, we have
\begin{equation} \label{eq_bps_S}
BPS_\beta^S(q)=\left(2 \sin \left(\frac{u}{2}\right) \right)^{1-m_\beta} \sum_{g \geq 0}
GW_{g,\beta}^S u^{2g-1+m_\beta} \,.
\end{equation}
\end{lemma}

\begin{proof}
We calculate $GW^X_{g,(\beta,0),\gamma_\beta}$
by applying the localization formula in Gromov--Witten theory \cite{GP_localization} to the $\CC^\star$ action on $X=S\times \PP^1$ which scales $\PP^1$ with the fixed points $0$ and $\infty$. 
To do this, we lift the insertion $\pt_{\PP^1}$ to the equivariant point class $[0]$. Since the restriction of $[0]$ to $\infty$ is zero, 
the contribution of the fixed locus consisting of stable maps mapping to $S \times \infty$ vanishes. On the other hand, the restriction of $[0]$ to $0$ is equal to the equivariant parameter $t$. As in \cite[Lemma 7]{MPT}, the contribution of the inverse Euler class of the virtual normal bundle to the fixed locus consisting of stable maps mapping to $S \times 0$ is given by $\frac{1}{t} (-1)^g \lambda_g$. Therefore, the factors $t$ and $\frac{1}{t}$ cancel, and so \eqref{eq_X_S} follows.
Finally, \eqref{eq_bps_S} follows immediately from the combination of \eqref{eq_X_S} and \eqref{Eq:alternativeBPS}.
\end{proof}

\begin{remark}
    Although \eqref{eq_bps_S} provides an explicit expression of $BPS_\beta^S(q)$ in terms of higher genus Gromov--Witten theory of $S$, it is unclear how to directly prove the integrality and polynomiality of $BPS_\beta^S(q)$ using this formula.
\end{remark}

The following result shows that the BPS polynomials $BPS_\beta^S(q)$ specialize at $q=1$ to the genus zero Gromov--Witten invariants $GW_{0,\beta}^S$ introduced in \S \ref{section_complex}.

\begin{corollary}
\label{cor_q1}
Let $S$ be a smooth projective surface over $\CC$ and $\beta \in H_2(S,\ZZ)$ such that 
$m_\beta:=-1+c_1(S)\cdot \beta \geq 0$. Then,  $BPS_\beta^S(1)=GW_{0,\beta}^S$.
\end{corollary}

\begin{proof}
This follows directly from taking the limit $q=e^{iu} \rightarrow 1$, that is, $u \rightarrow 0$, in \eqref{eq_bps_S} of Lemma \ref{lem_X_S}.
\end{proof}

\subsubsection{BPS polynomials and stable pairs}
\label{sec_stable_pairs}
For any $3$-fold $X$, the Gromov--Witten/pairs correspondence  \cite{MNOP1, MNOP2}\cite[Conjecture 3.28]{PT1}, established in many cases  
\cite{MOOP, PP, pardon2023universally}, predicts that the BPS invariants $BPS_{g,\beta,\gamma}^X$ of $X$ have an alternative description in terms of Pandharipande--Thomas invariants $PT_{\beta,\chi,\gamma_\beta}^X$ with insertions $\gamma_\beta$ defined using moduli spaces of stable pairs $(\cO_X \rightarrow F)$ with $[F]=\beta$ and $\chi(F)=\chi$ \cite{PT1}. 
When $X=S \times \PP^1$ and $m_\beta\geq 0$, the following result shows that the general Gromov--Witten/pairs correspondence simplifies.

\begin{lemma}\label{lemma_stable_pairs}
Let $S$ be a smooth projective surface over $\CC$ and $\beta \in H_2(S,\ZZ)$ such that 
$m_\beta:=-1+c_1(S)\cdot \beta \geq 0$. Then, the Gromov--Witten/pairs correspondence for $X=S \times \PP^1$ implies that
\[ \sum_{\chi \in \ZZ} PT_{\beta,\chi, \gamma_\beta}^{S \times \PP^1} (-q)^{\chi} =-q (1-q)^{m_\beta-1} BPS_\beta^S(q)\,.\]
\end{lemma}

\begin{proof}
The Gromov--Witten/pairs correspondence in \cite[Conjecture 3.28]{PT1} involves disconnected Gromov--Witten invariants. However, since the definition of the Gromov--Witten invariants $GW_{g,(\beta,0),\gamma_\beta}$
contains a single insertion of the pullback of a point class from $\PP^1$, it follows from the product formula in Gromov--Witten theory \cite{behrend_product} that the only non-zero disconnected invariants are the connected ones. Thus, using that $c_1(S)\cdot \beta=m_\beta +1$, \cite[Conjecture 3.28]{PT1} becomes in this case
\[ \sum_{\chi \in \ZZ} PT_{\beta,\chi, \gamma_\beta}^{S \times \PP^1} (-q)^{\chi} = q^{\frac{1}{2}(m_\beta+1)}(-i)^{m_\beta+1} \sum_{g \geq 0} GW_{g,\beta,\gamma_\beta}^X u^{2g-1+m_\beta} \,.\]
By \eqref{Eq:alternativeBPS}, we have 
\[
\sum_{g \geq 0} GW_{g,\beta, \gamma_\beta}^X u^{2g-1+m_\beta}
= ((-i)(q^{\frac{1}{2}}-q^{-\frac{1}{2}}))^{m_\beta-1}BPS_{\beta}^S(q) \,,\]
and so we obtain
\[ \sum_{\chi \in \ZZ} PT_{\beta,\chi, \gamma_\beta}^{S \times \PP^1} (-q)^{\chi} = q^{\frac{1}{2}(m_\beta+1)}(-i)^{m_\beta+1} ((-i)(q^{\frac{1}{2}}-q^{-\frac{1}{2}}))^{m_\beta-1}BPS_{\beta}^S(q)
= -q (1-q)^{m_\beta-1} BPS_\beta^S(q)\,.\]

\end{proof}

In this paper, we will focus on the case where $S=S_n$ is the blow-up of $\PP^2$ at $n$ points in general position. 
Since $X=S_n \times \PP^1$ is deformation equivalent to a toric $3$-fold, the Gromov--Witten/pairs correspondence holds in this case by \cite{MOOP}.

\subsubsection{BPS polynomials and blow-ups}

The following result determines the BPS polynomials associated to classes $\beta$ of exceptional curves, that is, $\beta=[E]$ for a curve $E \simeq \PP^1 \subset S$ with $E^2=-1$. By the adjunction formula, we have $c_1(S) \cdot \beta =1$, and so $m_\beta=0$.

\begin{lemma}
\label{lem_ex}
Let $S$ be a smooth projective surface and $\beta=[E] \in H_2(S,\ZZ)$ the class of an exceptional curve $E \subset S$. Then, the corresponding BPS polynomial is equal to $1$: \[BPS_\beta^S(q)=1\,.\]
\end{lemma}

\begin{proof}
Since the curve $E$ is rigid in $S$, it follows from the localization formula as in \cite[Lemma 7]{MPT} that we have 
\begin{equation} \label{eq_local_curve}
GW_{g,\beta}^S=t\, GW_{g,[\PP^1]}^{\cO_{\PP^1}\oplus \cO_{\PP^1}(-1)}\,,\end{equation}
where $GW_{g,[\PP^1]}^{\cO_{\PP^1}\oplus \cO_{\PP^1}(-1)}$ is the genus $g$ degree $1$ $\CC^\star$-equivariant Gromov--Witten invariant of the local curve $\cO_{\PP^1}\oplus \cO_{\PP^1}(-1)$, where $t$ is the equivariant parameter for the action of $\CC^\star$ scaling the fibers of the line bundle $\cO_{\PP^1}$. By degeneration of $\cO_{\PP^1}\oplus \cO_{\PP^1}(-1)$ into the normal crossing union of $\cO_{\PP^1}\oplus \cO_{\PP^1}$ and $\cO_{\PP^1}\oplus \cO_{\PP^1}(-1)$, and using \cite[Lemma 6.2]{BP_local_curves} and \cite[Lemma 6.3]{BP_local_curves} to evaluate the relative theories, we obtain that 
\[ \sum_{g \geq 0} GW_{g,[\PP^1]}^{\cO_{\PP^1}\oplus \cO_{\PP^1}(-1)} u^{2g-1}= \frac{1}{t}\frac{1}{2 \sin \left(\frac{u}{2}\right)}\,.\]
Hence, we have 
\[ \sum_{g \geq 0} GW_{g,\beta}^S u^{2g-1}=\frac{1}{2 \sin \left(\frac{u}{2}\right)}\,,\]
and so we obtain that $BPS_\beta^S(q)=1$ by \eqref{eq_bps_S} in Lemma \ref{lem_X_S}.
\end{proof}

Finally, the following result describes how the BPS polynomials change under blow-up. This will be used in the proof of Theorem \ref{thm_main} in \S \ref{section_bps_welschinger}.

\begin{lemma} \label{lem_blow_up}
Let $S$ be a smooth projective surface over $\CC$ and $\pi: \widetilde{S} \rightarrow S$ the blow-up of $S$ at a point. Then, for every $\beta \in H_2(S,\ZZ)$ and $g \in \ZZ_{\geq 0}$, we have $GW_{g, \pi^\star \beta}^{\widetilde{S}}=GW_{g,\beta}^S$. 
In particular, we have $BPS_{\pi^\star \beta}^{\widetilde{S}}(q)=BPS_{g,\beta}^S$.
\end{lemma}

\begin{proof}
For $g=0$, this is proved in \cite[Theorem 1.2]{Hu_blow_up}. 
For $g>0$ with insertion of $(-1)^g \lambda_g$, 
the proof of \cite[Theorem 1.2]{Hu_blow_up} applies without change thanks to 
the splitting/gluing property of $\lambda_g$ -- see for example
\cite[Lemma 7]{Bou2019}.
Finally, the result for the BPS polynomials follows from the result for the Gromov--Witten invariants by Lemma \ref{lem_X_S}.
\end{proof}

\subsection{BPS and Block--G\"ottsche polynomials}
\label{section_BPS_BG}
In this section, we show that the BPS polynomials of toric del Pezzo surfaces agree with the Block--G\"ottsche polynomials reviewed in \S\ref{section_BG}. To do this, we will use the main result of \cite{Bou2019}, which provides a 
description of the Block-G\"ottsche invariants 
using higher genus log Gromov--Witten invariants of $S$ relatively to the toric boundary, with insertion of $(-1)^g \lambda_g$. 
We refer to \cite{logGWbyAC, logGW} for foundations of log Gromov--Witten theory.
On the other hand, by Lemma \ref{lem_X_S}, the BPS polynomials can be expressed in terms of higher genus Gromov--Witten invariants of $S$, also with insertion of $(-1)^g \lambda_g$. Consequently, we will first establish a comparison between log and non-log (absolute) higher genus Gromov--Witten invariants of $S$ in Theorem \ref{thm_absolute_relative_0}. The comparison bewteen BPS and Block-G\"ottsche polynomials will follow in Theorem \ref{thm_absolute_relative}.

\subsubsection{Log and absolute Gromov--Witten invariants of toric del Pezzo surfaces} \label{section_log_absolute}

Let $S$ be a toric del Pezzo surface, with toric boundary divisor
$D=\sum_{i=1}^\ell D_i$. For every $\beta \in H_2(S,\ZZ)$ such that $\beta \cdot D_i \geq 0$ for every $1\leq i \leq \ell$, we denote by $GW_{g,\beta}^{S/D}$ the genus $g$ log Gromov--Witten invariant of $(S,D)$ of class $\beta$, 
with $\beta \cdot D_i$ (unordered) marked points having contact order one along $D_i$ for all $1\leq i\leq \ell$,
and with insertion of $(-1)^g \lambda_g$ and $m_\beta
=-1+\beta \cdot D$ point classes.

\begin{theorem} \label{thm_absolute_relative_0}
    Let $S$ be a smooth toric del Pezzo surface, and $\beta \in H_2(S,\ZZ)$
such that $\beta \cdot D_j \geq 0$ for every toric divisor $D_j$ of $S$. Then, for every $g \geq 0$, we have
 \begin{equation} \label{eq_log_non_log}
GW_{g,\beta}^S = GW_{g,\beta}^{S/D}\,.\end{equation}

\end{theorem}

The proof of Theorem \ref{thm_absolute_relative_0} takes the remainder of \S\ref{section_log_absolute}.
To do this, we consider the degeneration to the normal cone of $D$ in $S$. 
This degeneration is constructed explicitly using toric geometry as follows. 
For every ray $\rho$ of the fan $\Sigma_S$ of $S$ in $\RR^2$, denote by $u_\rho$ the primitive
integral point on $\rho$.
Let $\mathscr{P}_S$ be the polyhedral decomposition of $\RR^2$ obtained from $\Sigma_S$ by adding edges connecting the points $u_\rho$ together -- see Figure \ref{Fig9}. The polyhedral decomposition $\mathscr{P}_S$ determines as in 
\cite[\S 3]{NS}
a toric degeneration $\pi: \mathscr{S} \rightarrow \CC$, such that $\pi^{-1}(t)=S$ for $t \neq 0$ and the central fiber $\mathscr{S}_0 := \pi^{-1}(0)$ has dual intersection complex $\mathscr{P}_S$. In particular, the decomposition of $\mathscr{S}_0$ into irreducible components is given by 
\[ \mathscr{S}_0 =S \cup \bigcup_\rho \PP_\rho\,,\]
where $\PP_\rho$ are toric surfaces that are $\PP^1$-bundles over the divisors $D_\rho$ in $S$ corresponding to the rays $\rho$ of $\Sigma_S$.

\begin{figure}[htb]
\begin{center}
\includegraphics[height=1.6in,width=1.6in,angle=0]{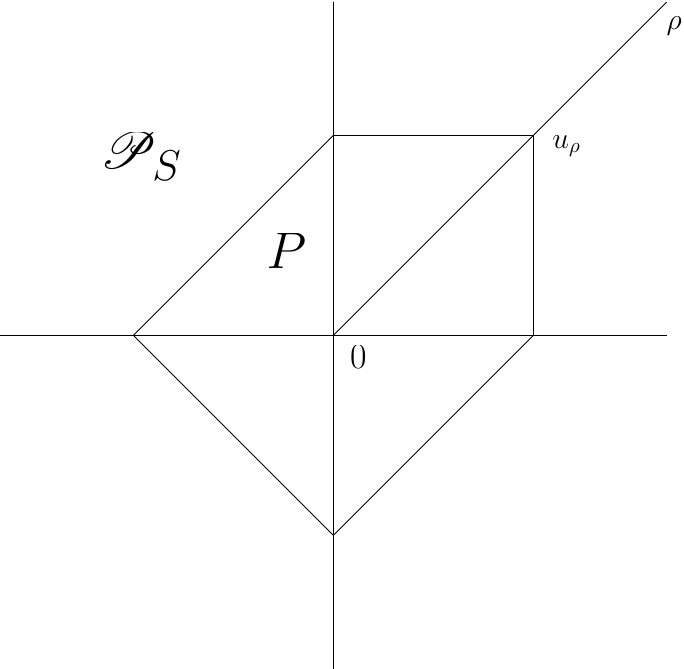}
\caption{The polyhedral decomposition $\mathscr{P}_S$.}
\label{Fig9}
\end{center}
\end{figure}

%\begin{figure}[htb]
%\begin{center}
%\includegraphics[height=1.8in,width=2in,angle=0]{Fig9.eps}
%\caption{The polyhedral decomposition $\mathscr{P}_S$.}
%\label{Fig9}
%\end{center}
%\end{figure}

We equip the total space $\mathscr{S}$ of the degeneration with the divisorial log structure defined by the central fiber $\mathscr{S}_0$, and $\CC$ with the divisorial log structure defined by $\{0\} \subset \CC$, so that the morphism $\pi:\mathscr{S} \rightarrow \CC$ naturally lifts to a log smooth morphism. The tropicalization of $\mathscr{S}$ is the cone over the compact polygon $P \subset \RR^2$ with vertices the points $u_\rho$. 
Finally, we specialize the $m_\beta$ point constraints for $t \neq 0$ to $m_\beta$ point constraints lying entirely within the irreducible component of the central fiber isomorphic to $S$.

By the decomposition formula in log Gromov--Witten theory of \cite[Theorem 5.4]{ACGS_decomposition}, applied to the log smooth degeneration $\pi: \mathscr{S} \rightarrow \CC$, we have 
\begin{equation} \label{eq_decomposition}
GW_{g,\beta}^S = \sum_{h: \Gamma \rightarrow P} \frac{n_h}{|\mathrm{Aut}(h)|} GW_h^{\mathscr{S}_0}\,,\end{equation}
where the sum is over the rigid decorated tropical curves $h:
\Gamma \rightarrow P$, of total genus $g$, total curve class $\beta$, with $m_\beta$ legs. 
The term $GW_h^{\mathscr{S}_0}$ is the log Gromov--Witten invariant of the central fiber $\mathscr{S}_0$ endowed with the restricted log structure from $\mathscr{S}$, given by 
\[ GW_h^{\mathscr{S}_0} = \int_{[\overline{M}_h(\mathscr{S}_0)]^{\vir}} (-1)^g \lambda_g \,,\]
where $\overline{M}_h(\mathscr{S}_0)$ is the moduli space of $h$-marked stable log maps to $\mathscr{S}_0$ passing through the $m_\beta$ point constraints imposed at the marked points corresponding to the $m_\beta$ legs of $\Gamma$ -- see \cite[Definition 2.31]{ACGS_decomposition}. Moreover, $n_h \in \ZZ_{\geq 1}$ is the smallest positive integer such that $n_h \cdot h(\Gamma)$ has integral vertices, and $|\mathrm{Aut}(h)|$ is the order of the group of automorphisms of $h$.
By the vanishing property of $\lambda_g$ reviewed in \cite[Lemma 8]{Bou2019}, we obtain that $GW_h^{\mathscr{S}_0}=0$ unless the graph $\Gamma$ is of genus zero. 

To calculate $GW_h^{\mathscr{S}_0}$ when the graph $\Gamma$ is of genus zero, we refine the polyhedral decomposition of $P$ so that it contains 
$h(\Gamma)$, and we consider the central fiber $\widetilde{\mathscr{S}}_0$ of the corresponding log modification of $\pi: \mathscr{S} \rightarrow \CC$. We denote by $\widetilde{S}^w$ the irreducible components of 
$\widetilde{\mathscr{S}}_0$ labeled by the vertices $w$ 
of the refined polyhedral decomposition. In particular, the irreducible component $\widetilde{S}^0$ corresponding to the origin is a toric blow-up of the irreducible component $S$ of $\mathscr{S}_0$.
We endow every irreducible component $\widetilde{S}^w$
with the divisorial log structure defined by the 
intersection $\partial \widetilde{S}^w$ of $\widetilde{S}^w$ with the singular locus of $\mathscr{S}_0$.

For every vertex $v$ of $\Gamma$, denote by $h_v: \Gamma_v \rightarrow P$ the rigid decorated tropical curve with one vertex obtained as the star of $v$ in $h: \Gamma \rightarrow P$. 
The image $h(v)$ is a vertex of the refined polyhedral decomposition, and so corresponds to an irreducible component $\widetilde{S}^{h(v)}$ of $\widetilde{\mathscr{S}}_0$. 
Let $\overline{M}_{h_v}(\widetilde{S}^{h(v)})$ be the moduli space of $h_v$-marked stable log maps
to $\widetilde{S}^{h(v)}$ passing through the point constraints imposed at the marked points corresponding to the legs of $\Gamma$ adjacent to $v$.
For every edge $e$ of $\Gamma$, denote by $\widetilde{D}^{h(e)}$ the irreducible component of the singular locus of $\widetilde{\mathscr{S}}_0$ corresponding to the edge $h(e)$ of the refined polyhedral decomposition of $P$.
We denote by $1_{h(e)} \in H^0(\widetilde{D}^{h(e)})$ the unit in cohomology and by $\mathrm{pt}^{h(e)} \in H^2(\widetilde{D}^{h(e)})$ the class of a point in $\widetilde{D}^{h(e)}$.

A \emph{splitting data} $\sigma$ assigns to each half-edge $(v,e)$ of $\Gamma$ a cohomology class $\sigma_{v,e}\in H^\star(\widetilde{D}_e)$, such that, for every edge $e$ adjacent to vertices $v$ and $v'$, 
we have either $\sigma_{v,e}=1_{h(e)}$ and $\sigma_{v',e}=\mathrm{pt}_{h(e)}$, or $\sigma_{v,e}=\mathrm{pt}_{h(e)}$ and 
$\sigma_{v',e}=1_{h(e)}$.
For every vertex $v$ of $\Gamma$, with genus decoration $g_v$, and for every splitting data $\sigma$, we define a log Gromov--Witten invariant 
\[ GW_{h,v,\sigma}=\int_{[\overline{M}_{h_v}(\widetilde{S}^{h(v)})]^{\vir}}(-1)^{g_v} \lambda_{g_v} \prod_{v\in e} \mathrm{ev}_e^\star(\sigma_{v,e})\,,\]
where the product is over the edges of $\Gamma$ adjacent to $v$, which are viewed as legs of $\Gamma_v$, and $\mathrm{ev}_e: \overline{M}_{h_v}(\widetilde{S}^{h(v)}) \rightarrow \widetilde{D}^{h(e)}$ are the evaluation morphisms at the corresponding marked points.

As in the proof of \cite[Proposition 13]{Bou2019}, we obtain the following gluing formula
\begin{equation}\label{eq_gluing}
GW_h^{\mathscr{S}_0}
= \sum_{\sigma} \prod_e w_e \prod_{v} GW_{h,v,\sigma} \,,\end{equation}
where the sum is over the splitting data $\sigma$, and $w_e$ are the weights of the edges of $\Gamma$.
We will now restrict the splitting data with possibly non-vanishing contributions using the dimension constraint for the log Gromov--Witten invariants $GW_{h,v,\sigma}$. In order to have $GW_{h,v,\sigma} \neq 0$, the dimension of the virtual class
$[\overline{M}_{h_v}(\widetilde{S}^{h(v)})]^{\vir}$ 
should match with the degree of the insertions, that is, 
\begin{equation}\label{eq_dim_constraint}
c_1(\widetilde{S}^{h(v)}) \cdot \beta_v + g_v -1+\sum_{v\in e} (1-w_e) + m_v = g_v + \sum_{v\in e} \mathrm{deg}_\CC\, \sigma_{v,e}\,,\end{equation}
where $m_v$ is the number of legs of $\Gamma$ adjacent to $v$, and $\mathrm{deg}_\CC\, \sigma_{v,e}$ denotes the complex degree of the cohomology class $ \sigma_{v,e}$.

To state the following result, we denote by $V_0(\Gamma)$ the set of vertices $v$ of $\Gamma$ such that $h(v)=0 \in P$, that is, such that $\widetilde{S}^{h(v)}=\tilde{S}^0$. We also denote by $\nu: \widetilde{\mathscr{S}}_0 \rightarrow S$ the morphism given by the composition of $\widetilde{\mathscr{S}}_0 \rightarrow \mathscr{S}_0$ with the natural projection
$\mathscr{S}_0 \rightarrow S$.

\begin{lemma} \label{lem_dim}
If $GW_{h,v,\sigma}\neq 0$, then the following holds:
\begin{itemize}
    \item[i)] For every vertex $v \notin V_0(\Gamma)$, we have $\nu_\star \beta_v=0$.
    \item[ii)] For every edge $e$ adjacent to a vertex $v \in V_0(\Gamma)$, the divisor $\widetilde{D}^{h(e)}$ is not an exceptional divisor of $\widetilde{S}^0 \rightarrow S$. Moreover, we have $w_e=1$ and $\sigma_{v,e}=1\in H^0(\widetilde{D}^{h(e)})$.
    \item[iii)] The set $V_0(\Gamma)$ consists of a single vertex.
\end{itemize}
\end{lemma}

\begin{proof}
If $GW_{h,v,\sigma}\neq 0$, then, summing the equalities given by \eqref{eq_dim_constraint}, we obtain that 
\begin{equation} \label{eq_dim_1}
\sum_{v \in V_0(\Gamma)} c_1(\widetilde{S}^0) \cdot \beta_v  = |V_0(\Gamma)|+m_\beta +\sum_{\substack{v\in V_0(\Gamma)\\ v\in e}} (\deg_\CC \sigma_{v,e} +w_e -1)\,.\end{equation}
For every $v\in V_0(\Gamma)$, denote by $Ex(v)$ the set of edges $e$ adjacent to $v$ such that $\widetilde{D}^{h(e)}$ is an exceptional divisor of $\widetilde{S}^0 \rightarrow S$. 
Since the surface $\widetilde{S}^0$ is a toric blow-up of $S$, we have 
\[ c_1(\widetilde{S}^0) \cdot \beta_v = c_1(S) \cdot \nu_\star \beta_v - \sum_{e \in Ex(v)} w_e \,.\]
Hence, \eqref{eq_dim_1} can be rewritten as:
\begin{equation} \label{eq_dim_2}
\sum_{v\in V_0(\Gamma)} c_1(S) \cdot \nu_\star \beta_v = |V_0(\Gamma)|+m +\sum_{\substack{v\in V_0(\Gamma)\\ v\in e}} (\deg_\CC\, \sigma_{v,e} +w_e -1) +\sum_{\substack{v\in V_0(\Gamma)\\ e\in Ex(v)}} w_e\,.
\end{equation}
Using that $\beta = \sum_v \nu_\star \beta_v$, and that we have $c_1(S) \cdot \beta = m_\beta+1$, we obtain:
\begin{equation} \label{eq_dim_3}
1 = \sum_{v \notin V_0(\Gamma)}c_1(S) \cdot \nu_\star \beta_v + |V_0(\Gamma)|+\sum_{\substack{v\in V_0(\Gamma)\\ v\in e}} (\deg_\CC \sigma_{v,e} +w_e -1) +\sum_{\substack{v\in V_0(\Gamma)\\ e\in Ex(v)}} w_e\,. \end{equation}
Since every class $\nu_\star \beta_v$ is effective and $S$ is a del Pezzo surface, we obtain that $c_1(S) \cdot \nu_\star \beta_v \geq 0$, with equality if and only if $\nu_\star \beta_v=0$. 
By \eqref{eq_positive_m}, we have $m_\beta \geq 1$, and so we necessarily have $|V_0(\Gamma)|\geq 1$ since the $m_\beta$ point constraints are imposed in 
$\widetilde{S}^0$. 
Since all the other terms in the right-hand side of \eqref{eq_dim_3} are nonnegative, 
we obtain that the following holds: $\nu_\star \beta_v=0$ for every $v\notin V_0(\Gamma)$, $|V_0(\Gamma)|=1$, $\mathrm{deg}_\CC\, \sigma_{v,e}=0$, $w_e=1$, and $Ex(v)=\emptyset$ for all $v\in V_0(\Gamma)$ and $v\in e$. This concludes the proof of Lemma \ref{lem_dim}.
\end{proof}

If $GW_{h,v,\sigma} \neq 0$, then, since $h:\Gamma \rightarrow P$ is a rigid tropical curve, and $\Gamma$ is a graph of genus zero, we obtain from Lemma \ref{lem_dim} that the following holds:
\begin{itemize}
    \item[i)] There exists a unique vertex $v_0$ of $\Gamma$ such that $h(v_0)=0$.
    \item[ii)] For every ray $\rho$ of $\Sigma_S$, there exist exactly $\beta \cdot D_\rho$ vertices $v_{\rho, j}$, $1\leq j \leq \beta \cdot D_\rho$, such that $h(v_{\rho,j})=u_\rho$.
    \item[iii)] Vertices of $\Gamma$ are exactly $v_0$ and $(v_{\rho,j})_{\rho,1\leq j \leq \beta \cdot D_\rho}$.
    \item[iv)] For every ray $\rho$ of $\Sigma_S$ and $1\leq j \leq \beta \cdot D_\rho$, there exists a unique edge $e_{\rho,j}$ connecting $v_0$ and $v_{\rho,j}$. Moreover, we have $w_{e_{\rho,j}}=1$, and $\Gamma$ does not contain any other edges.
    \item[v)] $\beta_{v_0}=\beta$, and, for every ray $\rho$ of $\Sigma_S$ and $1\leq j \leq \beta \cdot D_\rho$, the class $\beta_{v_{\rho,j}}$ is the class of a $\PP^1$-fiber of $\mathbb{P}_\rho \rightarrow D_\rho$.
\end{itemize}
In particular, the decorated tropical curve $h:\Gamma \rightarrow P$ is uniquely determined up to the genus decorations. Moreover, we have $\sigma_{v_0, e_{\rho,j}}=1_{u_\rho}$
and $\sigma_{v_{\rho,j}, e_{\rho,j}}=\mathrm{pt}_{u_\rho}$, and so the splitting data $\sigma$ is also uniquely determined. 
By \cite[Lemma 15]{Bou2019}, we have $GW_{h,v_{\rho,j},
\sigma}=0$ if $g_{v_{\rho,j}}>0$ and $GW_{h,v_{\rho,j},
\sigma}=1$ if $g_{v_{\rho,j}}=0$. Hence, the genus decoration is uniquely 
determined to be $g_{v_0}=g$ and $g_{v_{\rho,j}}=0$.
Consequently, the gluing formula \eqref{eq_gluing}
and the decomposition formula \eqref{eq_decomposition} 
reduce to 
\[ GW_{g,\beta}^S = GW_{g,\beta}^{S/D}\,,\]
and this concludes the proof of Theorem \ref{thm_absolute_relative_0}.

\subsubsection{BPS and Block-G\"ottsche polynomials of toric del Pezzo surfaces}

The following result shows that BPS polynomials recover  
Block-G\"ottsche polynomial for toric del Pezzo surfaces.

\begin{theorem} \label{thm_absolute_relative}
Let $S$ be a smooth toric del Pezzo surface, and $\beta \in H_2(S,\ZZ)$
such that $\beta \cdot D_j \geq 0$ for every toric divisor $D_j$ of $S$.
Then, the BPS polynomial $BPS_\beta^S(q)$ is equal to the Block-G\"ottsche polynomial $BG_\beta^S(q)$:
\[ BPS_\beta^S(q) = BG_\beta^S(q) \,.\]
\end{theorem}

\begin{proof}
By \cite[Theorem 1]{Bou2019}, we have
\begin{equation}\label{eq_BG}
BG_\beta^S(q)= \left( 2 \sin \left( \frac{u}{2} \right)\right)^{2-\beta \cdot D} \sum_{g \geq 0} GW_{g,\beta}^{S/D} u^{2g-2+\beta \cdot D}\,.\end{equation}
By Theorem \ref{thm_absolute_relative_0}, we have $GW_{g,\beta}^{S/D}=GW_{g,\beta}^S$, and so the result follows from
\eqref{eq_bps_S} in Lemma \ref{lem_X_S} describing the BPS polynomials $BPS_\beta^S(q)$ in terms of the Gromov--Witten invariants $GW_{g,\beta}^S$.
\end{proof}

\subsection{Towards Welschinger invariants from BPS polynomials at $q=-1$}
\label{section_conjecture}

Recall that, for every $n \in \ZZ_{\geq 0}$, we denote by $S_n$ a smooth projective surface over $\CC$ obtained by blowing up $n$ general points in $\PP^2$. 
We consider a curve class $\beta \in H_2(S_n,\ZZ)$ such that $m_\beta:=-1+c_1(S_n) \cdot \beta \geq 0$. 
In \S \ref{section_complex}, we reviewed
the definition of the Gromov--Witten count $GW_{0,\beta}^{S_n} \in \ZZ_{\geq 0}$ of complex rational curves in $S_n$ of class $\beta$ passing through $m_\beta$ points in general position.
When the $n$ blown-up points in $\PP^2$ are real,
$S_n$ is naturally a real surface and we described in 
\S \ref{section_real} the Welschinger count $W_\beta^{S_n} \in \ZZ$ of real rational curves in $S_n$ of class $\beta$ passing through $m_\beta$ real points in general position.

When $n \leq 3$, the surface $S_n$ is toric, allowing tropical geometry to be used as in \S \ref{section_BG} to define Block-G\"ottsche polynomials $BG_\beta^{S_n}(q) \in \ZZ_{\geq 0}[q^{\pm \frac{1}{2}}]$.
These polynomials have the remarkable property to interpolate between the complend the real Welschinger counts: we have 
$BG_\beta^{S_n}(1)=GW_{0,\beta}^{S_n}$ and $BG_\beta^{S_n}(-1)=W_\beta^{S_n}$ by \eqref{eq_BG_interpolation}. 

More generally, for any $n \in \ZZ_{\geq 0}$, Definition \ref{Defn:BPSS} introduces BPS polynomials $BPS_\beta^{S_n}(q) \in \ZZ[q^\pm]$, which by Theorem \ref{thm_absolute_relative} coincide with the Block-G\"ottsche polynomials $BG_\beta^{S_n}(q)$ when the latter are defined, that is for $n \leq 3$.
This naturally raises the question of whether the BPS polynomials continue to interpolate between the complex Gromov--Witten counts and the real Welschinger counts. By Corollary \ref{cor_q1}, we know that $BPS_\beta^{S_n}(1)=GW_{0,\beta}^{S_n}$. 
The relation with Welschinger invariants is more elusive since the BPS polynomials are defined using counts of higher genus complex curves with no reference to real geometry. Nevertheless, we conjecture that the specialization of the BPS polynomials at $q=-1$ recover the Welschinger invariants, thereby generalizing the interpolation property of Block--G\"ottsche polynomials:

\begin{conjecture} \label{conj_bps_w}
For every $n \in \ZZ_{\geq 0}$, let $S_n$ be the blow-up of $\PP^2$ at $n$ general real points, and $\beta \in H_2(S_n,\ZZ)$ be a curve class such that $m_\beta:=-1+c_1(S_n)\cdot \beta \geq 0$.
Then, the specialization at $q=-1$ of the BPS polynomial $BPS_\beta^{S_n}$ coincides with the Welschinger count $W_\beta^{S_n}$ of real rational curves in $S_n$ passing through $m_\beta$ real points in general position:
    \[ BPS_{\beta}^{S_n}(-1)=W_{\beta}^{S_n} \,.\]
\end{conjecture}

Conjecture \ref{conj_bps_w} holds when $n \leq 3$, that is when $S_n$ is toric, since the BPS polynomials coincide with the Block-G\"ottsche polynomials in this situation by Theorem \ref{thm_absolute_relative}. 
In Theorem \ref{thm_main}, we prove Conjecture \ref{conj_bps_w} for all $n \leq 6$, thereby providing strong evidence for the validity of the conjecture beyond the toric case in general. 

We also show below that Conjecture \ref{conj_bps_w} holds for any $n$ when $\beta$ is the class of a real exceptional curve, that is, $\beta=[E]$, where $E$ is a real curve in $S_n$ isomorphic to $\PP^1$ with its standard real structure, meaning $\PP^1(\RR)=\RR \PP^1$, and 
satisfying $E^2=-1$. 
In this case, by the adjunction formula, we have $c_1(S_n)\cdot \beta=1$  and so $m_\beta=0$.

\begin{theorem}
    For every $n \in \ZZ_{\geq 0}$, let $S_n$ be the blow-up of $\PP^2$ at $n$ general real points, 
    and $\beta \in H_2(S_n,\ZZ)$ a real exceptional curve class. Then, we have $GW_{0,\beta}^{S_n}=W_\beta^{S_n}=1$ and $BPS_\beta^{S_n}(q)$ is the constant polynomial equal to $1$. In particular, Conjecture \ref{conj_bps_w} holds for all real exceptional curve classes.
\end{theorem}

\begin{proof}
Since $\beta$ is a real exceptional curve class, there exists a unique rigid curve $E \simeq \PP^1$ of class $\beta$, and so $GW_{0,\beta}^{S_n}=1$. Moreover, the real locus of $E$ is isomorphic to $\RR \PP^1$, which does not contain any node. Thus, the Welschinger sign of $E$ is positive, and so we obtain $W_\beta^{S_n}=1$. Finally, we have $BPS_\beta^{S_n}(q)=1$ by Lemma \ref{lem_ex}.
\end{proof}

\section{Relative BPS polynomials and floor diagrams}
\label{sec: relative BPS}
Let $\mathfrak{C}$ be a smooth conic in $\PP^2$. For every 
$n \in \ZZ_{\geq 0}$, we denote by $\widetilde{S}_n$ the blow-up of $\PP^2$ at $n$ general points on $\mathfrak{C}$, and 
by $\widetilde{\mathfrak{C}}$ the strict transform of $\mathfrak{C}$ in $\widetilde{S}_n$. 

In this section, we study the enumerative geometry of curves in the pair $(\widetilde{S}_n, \widetilde{\mathfrak{C}})$. In \S \ref{sec:relative_bps}, we introduce higher genus relative Gromov--Witten invariants and relative BPS polynomials of $(\widetilde{S}_n, \widetilde{\mathfrak{C}})$. 
In \S \ref{sec_floor}, following \cite{BrugFloorconic}, we review the combinatorics of floor diagrams that describe
curves in a degeneration of $(\widetilde{S}_n, \widetilde{\mathfrak{C}})$. 
In \S \ref{sec_relative_floor}, we show that the higher genus relative Gromov--Witten invariants and relative BPS polynomials of $(\widetilde{S}_n, \widetilde{\mathfrak{C}})$
can be computed using refined counts of floor diagrams. Using this result, we prove in \S\ref{sec_relative_bps_w} that the specialization at $q=-1$ of the relative BPS polynomials is given by Welschinger counts of real rational curves in $(\widetilde{S}_n, \widetilde{\mathfrak{C}})$. Finally, explicit examples are presented in \S \ref{sec_examples}.

\subsection{Relative BPS polynomials}
\label{sec:relative_bps}

We define relative higher genus Gromov--Witten invariants of the pair $(\widetilde{S}_n,\widetilde{\mathfrak{C}})$.
Fix $\beta \in H_2(\widetilde{S}_n,\ZZ)$. 
Let $\mu = (\mu_j)_{1\leq j \leq \ell(\mu)}$
and $\nu=(\nu_j)_{1\leq j \leq \ell(\nu)}$
be two ordered partitions such that 
\begin{equation} \label{eq_mu_nu}
\sum_{i=1}^{\ell(\mu)} \mu_i + \sum_{j=1}^{\ell(\nu)} \nu_j =\beta \cdot \widetilde{\mathfrak{C}}\,. \end{equation}
Let 
\begin{equation}
\label{eq: m}
m_{\beta,(\mu,\nu)} = H \cdot \beta - 1 + \ell(\nu)    
\end{equation}
where $H$ is the pull-back in $\widetilde{S}_n$ of the class of a line in $\PP^2$, and denote by 
\[ \overline{M}_{g,m_{\beta,(\mu,\nu)}}(\widetilde{S}_n/\widetilde{\mathfrak{C}}, \beta,\mu,\nu) \,, \]
the moduli space of genus $g$ stable maps to $\widetilde{S}_n$ relative to $\widetilde{\mathfrak{C}}$, having
$m_{\beta,(\mu,\nu)}$ marked point with contact order zero along  
$\widetilde{\mathfrak{C}}$, $\ell(\mu)$ marked points with contact orders $(\mu_j)_{1\leq j \leq \ell(\mu)}$ along $\widetilde{\mathfrak{C}}$, and 
$\ell(\nu)$ marked points with contact orders $(\nu_j)_{1\leq j \leq \ell(\nu)}$ along $\widetilde{\mathfrak{C}}$.
We will impose point constraints on the $m$ points with contact order zero and the $\ell(\mu)$ points with contact orders $(\mu_j)_{1\leq j \leq \ell(\mu)}$.
For this, let 
\begin{align*}
ev_i^{\widetilde{S}_n} : \overline{M}_{g,m_{\beta,(\mu,\nu)}}(\widetilde{S}_n/\widetilde{\mathfrak{C}}, \beta,\mu,\nu) & \longrightarrow \widetilde{S}_n, \,\ \mathrm{for} \,1 \leq i \leq m\,,
\end{align*}
and 
\begin{align*}
ev_j^{\widetilde{\mathfrak{C}}} : \overline{M}_{g,m_{\beta,(\mu,\nu)}}(\widetilde{S}_n/\widetilde{\mathfrak{C}}, \beta,\mu,\nu) & \longrightarrow \widetilde{\mathfrak{C}}, \,\ \mathrm{for} \,1 \leq j \leq \ell(\mu)\,,
\end{align*}
be the natural evaluation maps. Denote by  $\mathrm{pt}_{\widetilde{S}_n} \in H^4(\widetilde{S}_n, \ZZ)$ 
(resp.\ 
$\mathrm{pt}_{\widetilde{\mathfrak{C}}} \in H^2(\widetilde{\mathfrak{C}},\ZZ)$)
be the Poincare dual of the class of a point in $\widetilde{S}_n$ (resp.\ $\widetilde{\mathfrak{C}}$). 
We consider the genus $g$ relative Gromov--Witten invariants of $(\widetilde{S}_n,\widetilde{\mathfrak{C}})$, defined by 
\begin{equation}
\label{def: gw}
GW_{g,\beta,(\mu,\nu)}^{\widetilde{S}_n/\widetilde{\mathfrak{C}}} := 
\frac{1}{|\mathrm{Aut}(\mu)|}
\frac{1}{|\mathrm{Aut}(\nu)|}
\int_{[\overline{M}_{g,m_{\beta,(\mu,\nu)}}(\widetilde{S}_n/\widetilde{\mathfrak{C}}, \beta,\mu,\nu)]^{\mathrm{vir}}} (-1)^g \lambda_g  \prod_{i=1}^{m_{\beta,(\mu,\nu)}} (ev_i^{\widetilde{S}_n})^*(\mathrm{pt}_{\widetilde{S}_n}) 
 \prod_{j=1}^{\ell(\mu)} (ev_j^{\widetilde{\mathfrak{C}}})^*(\mathrm{pt}_{\widetilde{\mathfrak{C}}}) \,,
\end{equation}
where $|\mathrm{Aut}(\mu)|$
(resp.\ $|\mathrm{Aut}(\nu)|$) is the order of the group of automorphisms of the ordered partition $\mu$ (resp.\ $\nu$).
When $g=0$, the Gromov--Witten invariants \eqref{def: gw} agree with the ones defined in \cite[\S 2.2]{BrugFloorconic}. 
%However, for $g> 0$ they differ, due to the additional $\lambda_g$ class insertions we have.     

By the following dimension calculation,
the degree of the integrand in \eqref{def: gw} equals the virtual dimension of the moduli space $\overline{M}_{g,m}(\widetilde{S}_n/\widetilde{\mathfrak{C}}, \beta,\mu,\nu)$.
Denoting by $E_1, \cdots, E_n$ the classes of the exceptional curves of the blow-up $\widetilde{S}_n \rightarrow \PP^2$, we have \[ c_1(\widetilde{S}_n)=3H -\sum_{i=1}^n E_i\,\,\, \text{and}\,\,\, \widetilde{\mathfrak{C}}=2H - \sum_{i=1}^n E_i\,.\]
Therefore, using \eqref{eq_mu_nu}, we obtain
\begin{equation} \label{eq_c1} 
c_1(\widetilde{S}_n)\cdot \beta -\sum_{i=1}^{\ell(\mu)} \mu_i - \sum_{j=1}^{\ell(\nu)}\nu_j
= c_1(\widetilde{S}_n)\cdot \beta -\widetilde{\mathfrak{C}}\cdot \beta = H \cdot \beta \,.\end{equation}

\begin{lemma}
The virtual dimension of the moduli space $\overline{M}_{g,m_{\beta,(\mu,\nu)}}(\widetilde{S}_n/\widetilde{\mathfrak{C}}, \beta,\mu,\nu)$ satisfies
\[ \mathrm{virdim}(\overline{M}_{g,m_{\beta,(\mu,\nu)}}(\widetilde{S}_n/\widetilde{\mathfrak{C}}, \beta,\mu,\nu))=g+2m_{\beta,(\mu,\nu)}+\ell(\mu)\,. \]
\end{lemma}

\begin{proof}
By \cite{li2002degeneration}, the virtual dimension of the moduli space of relative stable maps is given by
\[ \mathrm{virdim}(\overline{M}_{g,m}(\widetilde{S}_n/\widetilde{\mathfrak{C}}, \beta,\mu,\nu))=g-1+c_1(\widetilde{S}_n) \cdot \beta +\sum_{i=1}^{\ell(\mu)} (1-\mu_i)+ 
\sum_{j=1}^{\ell(\nu)}(1-\nu_j)+
m_{\beta,(\mu,\nu)}\,. \]
By \eqref{eq_c1}, this can be rewritten as
\[ \mathrm{virdim}(\overline{M}_{g,k+m}(\widetilde{S}_n/\widetilde{\mathfrak{C}}, \beta,\alpha))=g-1+H \cdot \beta +\ell(\mu)+\ell(\nu)+m_{\beta,(\mu,\nu)}\,,\]
which is equal to $g+2m_{\beta,(\mu,\nu)}+\ell(\mu)$ since $H \cdot \beta-1+\ell(\nu)=m_{\beta,(\mu,\nu)}$.
\end{proof}

We now define a relative version of the BPS invariants.

\begin{definition} \label{def_relative_bps_inv}
Let $\beta \in H_2(\widetilde{S}_n,\ZZ)$ such that $H \cdot \beta \geq 1$, and $\mu$, $\nu$ as in \eqref{eq_mu_nu}. 
For every $g \geq 0$, we define the \emph{relative BPS invariants} $BPS_{g,\beta,(\mu,\nu)}^{\widetilde{S}_n/\widetilde{\mathfrak{C}}} \in \QQ$ by the formula
\begin{align} \label{eq_relative_bps_inv}
&\sum_{g \geq 0} GW_{g,\beta,(\mu,\nu)}^{\widetilde{S}_n/\widetilde{\mathfrak{C}}} u^{2g-2+H \cdot \beta + \ell(\mu)+\ell(\nu)} \\
&= \left(\prod_{j=1}^{\ell(\mu)} \frac{1}{\mu_j} 2 \sin\left(\frac{\mu_j u}{2}\right)\right)
\left(\prod_{j=1}^{\ell(\nu)} \frac{1}{\nu_j} 2 \sin\left(\frac{\nu_j u}{2}\right)\right)
\sum_{g \geq 0} BPS_{g,\beta,(\mu,\nu)}^{\widetilde{S}_n/\widetilde{\mathfrak{C}}}  
\left( 2 \sin \left( \frac{u}{2} \right)\right)^{2g-2+H \cdot \beta} \,. \nonumber\end{align}
\end{definition}

\begin{definition} \label{def_relative_BPS}
Let $\beta \in H_2(\widetilde{S}_n,\ZZ)$ such that $H \cdot \beta \geq 1$, and $\mu$, $\nu$ as in \eqref{eq_mu_nu}. 
We define the \emph{relative BPS polynomial} $BPS_{\beta,(\mu,\nu)}^{\widetilde{S}_n/\widetilde{\mathfrak{C}}}(q)$
by 
\begin{align} \label{eq_relative_bps_polyn}
BPS_{\beta,(\mu,\nu)}^{\widetilde{S}_n/\widetilde{\mathfrak{C}}}(q)&:= 
\sum_{g \geq 0} BPS_{g,\beta,(\mu,\nu)}^{\widetilde{S}_n/\widetilde{\mathfrak{C}}}  
\left( 2 \sin \left( \frac{u}{2} \right)\right)^{2g}
= \sum_{g \geq 0} BPS_{g,\beta,(\mu,\nu)}^{\widetilde{S}_n/\widetilde{\mathfrak{C}}}  
(-1)^g (q^{\frac{1}{2}}-q^{-\frac{1}{2}})^{2g}\\
&=\sum_{g \geq 0} BPS_{g,\beta,(\mu,\nu)}^{\widetilde{S}_n/\widetilde{\mathfrak{C}}}  
(-1)^g (q-2+q^{-1})^g \,, \nonumber
\end{align}
where $q=e^{iu}$.
\end{definition}

Despite the name given in Definition \ref{def_relative_BPS}, it is not clear at this point that $BPS_{\beta,(\mu,\nu)}^{\widetilde{S}_n/\widetilde{\mathfrak{C}}}(q)$ is a Laurent polynomial in $q$. Nevertheless, we will show in Corollary \ref{cor_bps_floor} that $BPS_{\beta,(\mu,\nu)}^{\widetilde{S}_n/\widetilde{\mathfrak{C}}}(q)$ is indeed a Laurent polynomial in $q$ with integer coefficients.

\begin{remark}
    Arguing as in the proof of Lemma \ref{lem_X_S}, the relative Gromov--Witten invariants $GW_{g,\beta,(\mu,\nu)}^{\widetilde{S}_n/\widetilde{\mathfrak{C}}}$ of the surface $(\widetilde{S}_n,\widetilde{\mathfrak{C}})$ with insertion of $(-1)^g \lambda_g$ can be viewed as relative Gromov--Witten invariants of the 3-fold $(\widetilde{S}_n \times \PP^1, \widetilde{\mathfrak{C}} \times \PP^1)$.
    The structure of the formula \eqref{eq_relative_bps_inv} defining the relative BPS invariants, including the factors $\frac{1}{\mu_j} 2 \sin\left(\frac{\mu_j u}{2}\right)$ and $\frac{1}{\nu_j} 2 \sin\left(\frac{\nu_j u}{2}\right)$, agrees with the general BPS integrality prediction in the string theory literature for open Gromov--Witten invariants of 3-folds -- see for example \cite[Eq. (2.10)]{openbps}. Other examples of this higher genus relative/open integrality are discussed in 
    \cite[\S 1.5.1]{bousseau_tak}, \cite[\S 8]{bousseau_looijenga}, \cite{open_bps_toric} and \cite{bousseau2023all, guo2025poincar}.
\end{remark}

\subsection{Floor diagrams}
\label{sec_floor}

We begin by reviewing the definition of marked floor diagrams for $(\widetilde{S}_n, \widetilde{\mathfrak{C}})$, following \cite[\S3.1]{BrugFloorconic}, and discussing the enumeration of such diagrams with refined multiplicities. We then introduce the degenerations that will be utilized in the subsequent section, where we prove that refined counts of floor diagrams correspond to the higher
genus Gromov--Witten invariants introduced in \S\ref{sec:relative_bps}.

\subsubsection{A degeneration of $\widetilde{S}_n$}
\label{subsub: a degeneration}
Recall that we denote by $\widetilde{S}_n$ the blow-up of $\PP^2$ at $n$ points on a smooth conic $\mathfrak{C}$.
Since the normal bundle to $\mathfrak{C} \simeq \PP^1$ in $\PP^2$ is isomorphic to $\cO_{\PP^1}(4)$, we obtain by $(m+1)$ successive applications of the degeneration to the normal cone of $\mathfrak{C}$ a degeneration of $\PP^2$: 
\[ \epsilon: \mathcal{F} \to \mathbb{C} \,,\]
with central fiber 
\[\epsilon^{-1}(0) = \PP^2 \cup \mathbb{F}_4^{(m+1)} \cup \mathbb{F}_4^{(m)}  \cup  \ldots \cup \mathbb{F}_4^{(2)} \cup \mathbb{F}_4^{(1)} \,, \]
given by a union of $\PP^2$ with $m+1$ copies of Hirzebruch surfaces $\mathbb{F}_4^{(i)} \simeq \PP(\cO_{\PP^1} \oplus \cO_{\PP^1}(4))$, for $1\leq i \leq m+1$, glued pairwise along copies of $\mathfrak{C}$. 
By blowing-up $n$ sections of $\epsilon$ defining degenerations of the $n$ points in $\mathfrak{C}$ that we blow up to obtain $\widetilde{S}_n$, similarly as in \cite[\S5.3]{GPS}, we obtain a degeneration
\begin{equation}
    \label{eq:deg to normal}
    \widetilde{\epsilon}: \widetilde{\mathcal{F}} \to \mathbb{C} \,,
\end{equation}
with general fiber $\widetilde{S}_n$, and central fiber 
\[ \widetilde{\epsilon}^{-1}(0) = \PP^2 \cup\mathbb{F}_4^{(m)}  \cup  \ldots \cup \mathbb{F}_4^{(2)} \cup \mathbb{F}_4^{(1)}  \cup  \mathrm{Bl}_{n}\mathbb{F}_4 \,, \]
the union of $\PP^2$ with $m$ copies of $ \mathbb{F}_4^{(i)} $ and a final irreducible component given by the blow up of $ \mathbb{F}_4 $ along the limits of the $n$ points on $\mathfrak{C}$, as illustrated in Figure \ref{Fig6}.

In \S \ref{subsec_relative_floor}, we will calculate 
the Gromow--Witten invariants $GW_{g,\beta}^{\widetilde{S}_n/\widetilde{\mathfrak{C}}}$, defined as in \eqref{def: gw} with $m_{\beta,(\mu,\nu)}$ point insertions, using a degeneration of $\widetilde{S}_n$ as in \eqref{eq:deg to normal} with $m=m_{\beta,(\mu,\nu)}$, such that the $i$'th point insertion degenerates into the $i$'th copy of the Hirzebruch surface $\mathbb{F}_4$ in $\widetilde{\epsilon}^{-1}(0)$, as illustrated in Figure \ref{Fig6}.
We first describe the combinatorics of the curves in $\widetilde{\epsilon}^{-1}(0)$, obtained as degenerations of curves in $\widetilde{S}_n$ contributing to such counts, in terms of marked floor diagrams in the following section.

\subsubsection{Marked floor diagrams}
\label{subsec: marked floor}
We briefly review below the definition of floor diagrams, and marked floor diagrams following \cite[\S3.1]{BrugFloorconic}.

A \emph{weighted oriented graph} $\Gamma$ is a connected graph with finitely many vertices, edges adjacent to two vertices, and legs adjacent to a single vertex, with an orientation and a choice of positive integers $w_e$ and $w_l$, for each edge $e$ and each leg $l$, called the weight of $e$ and $l$ respectively. We use the notations $V(\Gamma)$, $E(\Gamma)$, and $L(\Gamma)$ to denote the sets of vertices, edges and legs of a graph $\Gamma$. Note that the set 
\[ V(\Gamma) \cup E(\Gamma) \cup L(\Gamma) \]
admits a natural \emph{partial ordering} of its elements, generated by the relations $e \leq v$ if $e \in E(\Gamma)$ is an edge (resp. $l \leq v$ if $l \in L(\Gamma)$ is a leg) adjacent to $v$ and oriented towards $v$, and $e \geq v$ if $e$ is an edge (resp. $l \geq v$ if $l \in L(\Gamma)$ is a leg) adjacent to $v$ and oriented away from $v$.

Furthermore, for each vertex $v \in V(\Gamma)$, we denote by $\mathrm{div}(v)$ the \emph{divergence of $v$}, defined as the sum of all the weights on all edges and legs oriented towards $v$ minus the sum of all the weights on all edges and legs oriented outwards from $v$. 

\begin{definition}
\label{Def:floor diagram}
A \emph{floor diagram} of genus $g_0 \in \ZZ_{\geq 0}$ and degree $d \in \ZZ_{\geq 1}$ is  a weighted oriented graph $\Gamma$ satisfying the following conditions:
\begin{itemize}
    \item[i)] The oriented graph $\Gamma$ is acyclic, that is, it does not contain any oriented cycles,
    \item[ii)] The first Betti number $\Gamma$ equals $g_0$,
\item[iii)] The weights on the legs of $\Gamma$ satisfy the equation
\[  \sum_{l \in L(\Gamma)} w_l = 2d \,\]
\item[iv)] All legs are oriented towards the vertex they are adjacent to,
\item[v)] For each vertex $v \in V(\Gamma)$, either $\mathrm{div}(v)=2$ or $\mathrm{div}(v)=4$. Moreover, if $\mathrm{div}(v)=2$, then all edges and legs adjacent to $v$ are oriented towards $v$.
\end{itemize}
\end{definition}
Throughout this paper all floor diagrams we work with will be of genus $g_0=0$.
We will use the notation 
\[ V_2(\Gamma)=\{v \in V(\Gamma)\,|\, \mathrm{div}(v)=2\}\,\,\, \text{and} \,\,\,V_4(\Gamma)=\{v \in V(\Gamma)\,|\, \mathrm{div}(v)=4\},\]
so that $V(\Gamma)=V_2(\Gamma) \amalg V_4(\Gamma)$.

\begin{remark}
\label{rem: div 2 vertices}
When drawing floor diagrams in the following sections, following \cite{BrugFloorconic} we plot vertices $v \in V_2(\Gamma)$ in gray disks and vertices $v \in V_2(\Gamma)$ in white disks. 
If $v \in V_2(\Gamma)$, then since by item v) of Definition \ref{Def:floor diagram} all edges of $v$ are oriented towards $v$, locally around the vertex of $v$, we can have either a single edge or leg with weight $2$, or two edges or legs each with weight $1$ -- see Figure \ref{Fig1}. We denote by $V_2^{(2)}(\Gamma)$ the set of $v\in V_2(\Gamma)$ adjacent to a single edge of weight 2, as on the left of Figure \ref{Fig1}, and by $V_2^{(1,1)}(\Gamma)$ the set of $v\in V_2(\Gamma)$ adjacent to two edges of weight one, as on the right of Figure \ref{Fig1}, so that 
\[ V_2(\Gamma)=V_2^{(2)}(\Gamma) \amalg V_2^{(1,1)}(\Gamma) \,.\]
By convention, when plotting floor diagrams, we label each edge and leg with its weight if that weight is strictly greater than $1$. An unlabeled edge or leg implies a weight of $1$.
\end{remark}

%\begin{figure}[hbt!]
%\center{\scalebox{.5}{\input{Fig1.pspdftex}}}
%\caption{Two floor diagrams of degree $2$}
%\label{Fig1}
%\end{figure}

\begin{figure}[htb]
\begin{center}
\includegraphics[height=0.8in,width=2in,angle=0]{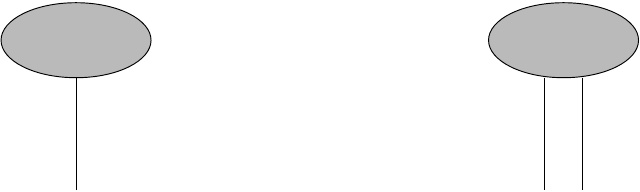}
\caption{Two floor diagrams of degree $2$}
\label{Fig1}
\end{center}
\end{figure}

We will later make use of the following elementary result on floor diagrams in the proof of Theorem \ref{thm_floor_gw}:

\begin{lemma} \label{lem_floor}
A floor diagram $\Gamma$ of degree $d$ satisfies
$|V_2(\Gamma)|+2|V_4(\Gamma)|=d$.
\end{lemma}

\begin{proof}
We evaluate $\sum_{v\in V(\Gamma)} \mathrm{div}(v)$ in two different ways. Since $V(\Gamma)=V_2(\Gamma)\amalg V_4(\Gamma)$, we obtain $\sum_{v\in V(\Gamma)} \mathrm{div}(v)=2|V_2(\Gamma)|+4|V_4(\Gamma)|$. On the other hand, expressing $\mathrm{div}(v)$ as a difference of weights, all terms cancel except the contributions of the legs. 
This yields $\sum_{v\in V(\Gamma)} \mathrm{div}(v)= \sum_{l \in L(\Gamma)} w_l$. 
By Definition \ref{Def:floor diagram}iii), this sum equates $2d$, competing the proof.
\end{proof}

Recall that we denote by $\widetilde{S}_n$ the blow-up of $\PP^2$ along $n$ distinct points on a conic $\mathfrak{C}$, and by $\widetilde{\mathfrak{C}}$ the strict transform of $\mathfrak{C}$. 
We let $E_1, \ldots E_n$ denote the exceptional curves, and $H$ the pull-back in $\widetilde{S}_n$ of the class of a line in $\PP^2$.

\begin{definition}
\label{def: marking on floor}
Let $\beta \in H_2(\widetilde{S}_n,\ZZ)$ such that $H \cdot \beta \geq 1$, and $\mu$, $\nu$ as in \eqref{eq_mu_nu}.
A \emph{marking of class $\beta$ and type $(\mu,\nu)$} of a floor diagram $\Gamma$ of genus $g_0$ and degree $d=H \cdot \beta$, is given by the following data:
\begin{itemize}
    \item[i)] A decomposition of the set of legs
    \[ L(\Gamma) = L_\mu \amalg L_\nu \amalg A_1 \amalg \ldots \amalg A_n \,,  \]
    where $L_\mu$ (resp.\ $L_\nu$) is a set of $\ell(\mu)$ (resp.\ $\ell(\nu)$) legs with weights $(\mu_i)_{1\leq i \leq \ell(\mu)}$
    (resp.\ $(\nu_i)_{1\leq i \leq \ell(\nu)}$), and for any $1\leq j \leq n$,  $A_j$ is a set of $\beta \cdot E_j$ legs of weight $1$, which are adjacent to distinct vertices of $\Gamma$.
 \item[ii)] A bijection 
 \begin{equation}
   \label{eq:phi}
   \phi:  \{ 1,\ldots m_{\beta,(\mu,\nu)} \} \xrightarrow{\sim }  V_4(\Gamma) \cup E(\Gamma) \cup L_\nu,
 \end{equation}
 where $m_{\beta,(\mu,\nu)}:=H \cdot \beta -1 +\ell(\nu)$, which is \emph{increasing}, that is,
 \[  \phi(i) > \phi(j) \implies i>j \, \]
with respect to the partial ordering on $V_4(\Gamma) \cup E(\Gamma) \cup L_\nu$, as described at the beginning of \S\ref{subsec: marked floor}.
\end{itemize}
A \emph{marked floor diagram} is a floor diagram equipped with a marking.
\end{definition}

To emphasize the decomposition of the legs in Definition \ref{def: marking on floor}, when drawing a floor diagram $\Gamma$ with $\mu=\emptyset$, we plot the legs in $L_{\nu}$ with black, and the ones in $A_1 \amalg \ldots \amalg A_n$ in red. 
The images of the numbers ${1,\ldots, m_{\beta,(\mu,\nu)}}$ under the bijection $\phi$ in Definition \ref{def: marking on floor} are displayed in red, to distinguish them from the weights on the edges. 

\begin{definition}
\label{def: marked floor isomorphic}
    Two marked floor diagram are \emph{isomorphic} if there exists an isomorphism between the underlying weighted oriented graphs, which preserves the decompositions of the sets of legs and commutes with the increasing bijections.
\end{definition}

\begin{example}
\label{Ex: d=2 simple}
In the floor diagrams illustrated in Figure \ref{Fig2}, there is a unique marking 
up to isomorphism. On both figures, there is 
up to isomorphism  a unique decompposition of the legs, 
\[ L(\Gamma) = L_{(1,1)} \amalg A_1 \amalg \ldots A_6 \,,  \]
where each of the $6$ red legs corresponds to an element in $A_i$, for $1\leq i \leq 6$. We have $m_{\beta,(\mu,\nu)}=4-1+1+1=5$, and the increasing bijection $\phi: \{ 1,\ldots, 5 \} \to V_4(\Gamma) \cup E(\Gamma) \cup L_\nu$ is also uniquely determined, up to isomorphism.

\begin{figure}[htb]
\begin{center}
\includegraphics[height=1.2in,width=4in,angle=0]{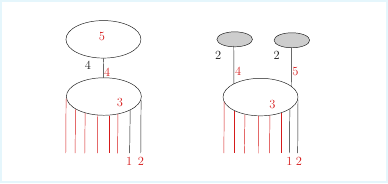}
\caption{Two floor diagrams of degree $4$ of class $\beta = 4H - \sum_{i=1}^6 E_i \in H_2(\widetilde{S}_6,\ZZ)$ and type $(\emptyset, (1,1))$.}
\label{Fig2}
\end{center}
\end{figure}

%\begin{figure}[hbt!]
%\center{\scalebox{.5}{\input{Fig2.pspdftex}}}
%\caption{Two floor diagrams of degree $4$ of class $\beta = 4H - \sum_{i=1}^6 E_i \in H_2(\widetilde{S}_6,\ZZ)$ and type $(\emptyset, (1,1))$.}
%\label{Fig2}
%\end{figure} 
\end{example}
We provide examples of floor diagrams on which there are several choices of markings in the following section.

\subsubsection{Curves in the degeneration of $\widetilde{S}_n$ and floor diagrams}
\label{subsec: curves and floors}
In this section we provide examples of marked floor diagrams and briefly explain how they encode topological information about curves in the
central fiber 
\[ \widetilde{\epsilon}^{-1}(0) = \PP^2 \cup \mathbb{F}_4^{(m_{\beta,(\mu,\nu)})} \cup \ldots \mathbb{F}_4^{(1)} \cup \mathrm{Bl}_n(\mathbb{F}_4) \, \]
of the
degeneration $\widetilde{\epsilon}$ 
of $\widetilde{S}_n$ described in \S\ref{subsub: a degeneration}. We refer to \cite{BrugFloorconic} for details.

\begin{figure}[htb]
\begin{center}
\includegraphics[height=1.2in,width=5in,angle=0]{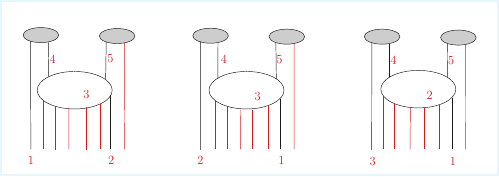}
\caption{Different markings on a floor diagram}
\label{Fig3}
\end{center}
\end{figure}

%\begin{figure}[hbt!]
%\center{\scalebox{.5}{\input{Fig3.pspdftex}}}
%\caption{Different markings on a floor diagram}
%\label{Fig3}
%\end{figure}

There exists a natural correspondence between the topology of irreducible components of curves in the central fiber, and vertices, edges, and legs of floor diagrams. 
White vertices marked by $i$, for $1\leq i \leq m_{\beta,(\mu,\nu)}$, correspond to curves in $\mathbb{F}_4^{(i)}$ of class $C_{-4} + bF$ with $b\in \ZZ_{\geq 0}$, where $F$ is the fiber class of $\mathbb{F}_4$ and $C_{-4}$ is the class of the $(-4)$-curve. 
Grey vertices represent lines in $\PP^2$. Bounded edges of weight $w$ correspond to chains of $\PP^1$ fibers of class $w F$. If a bounded edge which is marked by $k$ connects vertices marked by $i$ and $j$, for $i\leq k \leq j$, then the corresponding curve is a chain of $j-i$ copies of $\PP^1$'s.

\begin{figure}[htb]
\begin{center}
\includegraphics[height=1.5in,width=4.0in,angle=0]{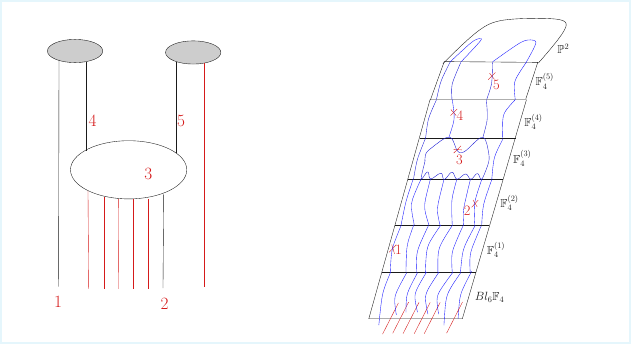}
\caption{A floor diagram with $m_{\beta,(\mu,\nu)}=5$, and the corresponding components of a curve in the central fiber $\widetilde{\epsilon}^{-1}(0)$. }
\label{Fig6}
\end{center}
\end{figure}

%\begin{figure}[hbt!]
%\center{\scalebox{.4}{\input{Fig6.pspdftex}}}
%\caption{A floor diagram with $m_{\beta,(\mu,\nu)}=5$, and the corresponding components of a curve in the central fiber $\widetilde{\epsilon}^{-1}(0)$. }
%\label{Fig6}
%\end{figure} 

\begin{example}
The floor diagram illustrated in Figure \ref{Fig3} admits $60$ different markings, since there are $6$ different choices, up to permutation, of the decomposition of the red legs, each corresponding to an element of $A_1, \ldots, A_6$. For each decomposition, there are $10$ possible choices of markings $\phi: \{ 1,\ldots , 5 \} \to V_4(\Gamma) \cup E(\Gamma) \cup L_\nu$: in the first floor diagram in Figure \ref{Fig3}, swapping $\phi(4)$ and $\phi(5)$ provides us with two different markings, similarly in the middle diagram we have two different markings obtained by swapping $\phi(4)$ and $\phi(5)$. Finally, in the right hand diagram, permuting $\phi(3),\phi(4)$, and $\phi(5)$ 
produces $6$ different markings. 
In Figure \ref{Fig6}, we illustrate the irreducible components of a curve in the central fiber corresponding to one of these marked floor diagrams. 
\end{example}

\subsection{Relative BPS polynomials from refined counts of floor diagrams}
\label{sec_relative_floor}

In this section, we prove in Theorem \ref{thm_floor_gw} and Corollary \ref{cor_bps_floor} 
that the higher genus relative Gromov--Witten invariants and relative BPS polynomials of $(\widetilde{S}_n, \widetilde{\mathfrak{C}})$
can be computed using refined counts of floor diagrams.
To do this, we first calculate the relative BPS polynomials
explicitly in two simple cases in \S \ref{sec_lines_2} and \S\ref{sec_line_tangent}. These results are used in the proof of the general case in \S\ref{subsec_relative_floor}.

\subsubsection{Lines intersecting the conic in two points}
\label{sec_lines_2}
In this section, we compute in a particular case the relative Gromov--Witten invariants $GW_{g,\beta, (\mu,\nu)}^{\widetilde{S}_n/\widetilde{\mathfrak{C}}}$ defined in \eqref{def: gw}. We assume that $n=0$, that is $\widetilde{S}_n=\PP^2$ and $\widetilde{\mathfrak{C}}=\mathfrak{C}$. Moreover, we assume that $\beta=H$, $\mu=(1,1)$ and $\nu=\emptyset$, that is, we are considering stable maps to $
\PP^2$ whose image is the line passing through two given distinct points on the conic $\mathfrak{C}$. In this case, 
there are no marked point with contact order zero since 
$m_{\beta,(\mu,\nu)}=H \cdot \beta -1+\ell(\nu)=1-1+0=0$.
The corresponding relative Gromov--Witten invariants $GW_{g,H, ((1,1),\emptyset)}^{\PP^2/\mathfrak{C}}$ are calculated by the result below.

\begin{lemma} \label{lem_gw_11}
We have 
\[ \sum_{g \geq 0} GW_{g,H,((1,1),\emptyset)}^{\PP^2/\mathfrak{C}} u^{2g} 
=u^{-1 }2 \sin \left( \frac{u}{2}\right) =u^{-1} (-i)(q^{\frac{1}{2}}-q^{-\frac{1}{2}})\,,\]
where $q=e^{iu}$. In particular, we have $BPS_{H,((1,1),\emptyset)}^{\PP^2/\mathfrak{C}}(q)=1$.
\end{lemma}

\begin{proof}
Let $GW_{g,H}^{\PP^2}$ be the $2$-pointed genus $g$ Gromov--Witten invariant of $\PP^2$ of class $H$, with insertion of $(-1)^g \lambda_g$ and of two point constraints.
By Theorem 3.4, $GW_{g,H}^{\PP^2}$ equals the corresponding log Gromov--Witten invariant of $\PP^2$ endowed with its toric boundary. 
There exists a unique tropical line passing through two distinct point in $\RR^2$, with a single 3-valent vertex of multiplicity $1$. Thus, the Block-G\"ottsche refined count of lines in $\PP^2$ passing by two distinct points is the constant polynomial $1$. Therefore, by \cite[Theorem 1]{Bou2019}, we obtain that 
\begin{equation} \label{eq_gw_11}
\sum_{g \geq 0} GW_{g,H}^{\PP^2} u^{2g-1} = 2 \sin \left( \frac{u}{2}\right) = (-i)(q^{\frac{1}{2}}-q^{-\frac{1}{2}})\,.\end{equation}
On the other hand, we obtain a different expression for $GW_{g,H}^{\PP^2}$ using the degeneration of $\PP^2$ to the normal cone of the conic $\mathfrak{C}$, with central fiber $\PP^2 \cup \mathbb{F}_4$. 
By degenerating the two point insertions to points in $\mathbb{F}_4$, 
the degeneration formula in relative Gromov--Witten theory \cite{li2002degeneration} 
implies that 
\[ \sum_{g \geq 0} GW_{g,H}^{\PP^2} u^{2g-2}=  \left(\sum_{g \geq 0} GW_{g,H, ((1,1),\emptyset)}^{\PP^2/\mathfrak{C}} u^{2g} \right) 
\left( \sum_{g\geq 0} GW_{g,F}^{\mathbb{F}_4/\mathfrak{C}_{-4}} u^{2g-1}\right)^2\,.\]
Here, $GW_{g,F}^{\mathbb{F}_4/\mathfrak{C}_{-4}}$ is the $1$-pointed genus $g$ Gromov--Witten invariant of $\mathbb{F}_4$ relative to $\mathfrak{C}_{-4}$, of class $F$ and with insertion of $(-1)^g \lambda_g$ and of a point constraint, see Figure \ref{Fig4}. 
As in \cite[Lemma 14]{Bou2019}, the Mumford's relation $\lambda_g^2=0$ when $g>0$ implies that $GW_{g,F}^{\mathbb{F}_4/\mathfrak{C}_{-4}}=1$ when $g=0$ and  $GW_{g,F}^{\mathbb{F}_4/\mathfrak{C}_{-4}}=0$ when $g>0$.
Consequently, we obtain
\[ \sum_{g \geq 0} GW_{g,H}^{\PP^2} u^{2g-2}=  \sum_{g \geq 0} GW_{g,H,((1,1),\emptyset)}^{\PP^2/\mathfrak{C}} u^{2g-2} \]
and so Lemma \ref{lem_gw_11} follows from Equation \eqref{eq_gw_11}.
\end{proof}

\begin{figure}[htb]
\begin{center}
\includegraphics[height=1.2in,width=5in,angle=0]{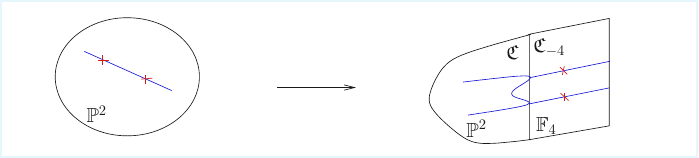}
\caption{The degeneration used in the proof of Lemma \ref{lem_gw_11}.}
\label{Fig4}
\end{center}
\end{figure}

%\begin{figure}[hbt!]
%\center{\scalebox{.7}{\input{Fig4.pspdftex}}}
%\caption{Degeneration used in the proof of Lemma \ref{lem_gw_11}.}
%\label{Fig4}
%\end{figure} 

\subsubsection{Lines tangent to the conic}
\label{sec_line_tangent}

In this section, we compute another particular case of the relative Gromov--Witten invariants $GW_{g,\beta,(\mu,\nu)}^{\widetilde{S}_n/\widetilde{\mathfrak{C}}}$ defined in \eqref{def: gw}. 
We assume that $n=0$, that is $\widetilde{S}_n=\PP^2$ and $\widetilde{\mathfrak{C}}=\mathfrak{C}$. Moreover, we assume that $\beta=H$, $\mu=(2)$ and $\nu=\emptyset$, that is, we are considering stable maps to $
\PP^2$ whose image in is the line tangent at a given point to the conic $\mathfrak{C}$. 
In this case, there no marked point with contact order zero since 
$m_{\beta,(\mu,\nu)}=H \cdot \beta -1+\ell(\nu)=1-1+0=0$.
The corresponding relative Gromov--Witten invariants $GW_{g,H, ((2),\emptyset)}^{\PP^2/\mathfrak{C}}$ are calculated by the result below.

\begin{lemma}
\label{lem_gw_2}
We have 
\[ \sum_{g \geq 0} GW_{g,H, ((2),\emptyset)}^{\PP^2/\mathfrak{C}} u^{2g-1} 
=u^{-1}\cos \left( \frac{u}{2}\right) =u^{-1} \frac{1}{2}(q^{\frac{1}{2}}+q^{-\frac{1}{2}})=u^{-1}\frac{[2]_q}{2}\,,\]
where $q=e^{iu}$ and 
\[[2]_q :=
\frac{q^{\frac{2}{2}}-q^{-\frac{2}{2}}}{q^{\frac{1}{2}}-q^{-\frac{1}{2}}}
=q^{\frac{1}{2}}+q^{-\frac{1}{2}}\] 
is the $q$-integer version of $2$ -- see \eqref{eq_q_integer}. In particular, we have $BPS_{H, ((2),\emptyset)}^{\PP^2/\mathfrak{C}}(q)=1$.
\end{lemma}

\begin{proof}
Let $GW_{g,H}^{\PP^2,\psi}$ be the $1$-pointed genus $g$ Gromov--Witten invariant of $\PP^2$ of class $H$, with insertion of $(-1)^g \lambda_g$, a point constraint and a psi-class. Arguing as in the proof of Theorem \ref{thm_absolute_relative}, we obtain that 
$GW_{g,H}^{\PP^2,\psi}$ equals the corresponding log Gromov--Witten invariant of $\PP^2$ endowed with its toric boundary.
Hence, it follows from \cite[Theorem A]{refined_descendants} that
\begin{equation} \label{eq_gw_2}
\sum_{g \geq 0} GW_{g,H}^{\PP^2,\psi}u^{2g} = \cos\left( \frac{u}{2}\right)=\frac{1}{2}(q^{\frac{1}{2}}+q^{-\frac{1}{2}})\,.\end{equation}
On the other hand, we obtain a different expression for $GW_{g,H}^{\PP^2, \psi}$ using the degeneration of $\PP^2$ to the normal cone of the conic $\mathfrak{C}$, with central fiber $\PP^2 \cup \mathbb{F}_4$. 
By degenerating the point insertion to a point in $\mathbb{F}_4$, the degeneration formula in relative Gromov--Witten theory \cite{li2002degeneration}
implies that 
\begin{equation}\label{eq_gw_2_2}
\sum_{g \geq 0} GW_{g,H}^{\PP^2, \psi} u^{2g-2}=  2\left(\sum_{g \geq 0} GW_{g,H, ((2),\emptyset)}^{\PP^2/\mathfrak{C}} u^{2g-1} \right) 
\left( \sum_{g\geq 0} GW_{g,2F}^{\mathbb{F}_4/\mathfrak{C}_{-4}, \psi} u^{2g-1}\right)\,.\end{equation}
Here, $GW_{g,2F}^{\mathbb{F}_4/\mathfrak{C}_{-4},\psi}$ is the $1$-pointed genus $g$ Gromov--Witten invariant of $\mathbb{F}_4$ of class $2F$, relative to $\mathfrak{C}_{-4}$ at one point of contact order $2$, and with insertion of $(-1)^g \lambda_g$, a point constraint and a psi class, see Figure \ref{Fig5}. 
As in \cite[Lemma 14]{Bou2019}, the Mumford's relation $\lambda_g^2=0$ when $g>0$ implies that $GW_{g,2F}^{\mathbb{F}_4/\mathfrak{C}_{-4},\psi}=0$ when $g>0$.
On the other hand, taking the coefficient of $u^{-2}$ in \eqref{eq_gw_2_2}, we obtain
\[ GW_{0,H}^{\PP^2, \psi} = 2 \,GW_{0,H, ((2),\emptyset)}^{\PP^2/\mathfrak{C}}\,  GW_{0,2F}^{\mathbb{F}_4/\mathfrak{C}_{-4}, \psi}\,. \]
Since there is a unique line passing though a given point with given tangent line, we have $GW_{0,H}^{\PP^2, \psi}=1$. Moreover, since there is a unique line tangent to a given conic at a given point, we also have $GW_{0,H, ((2),\emptyset)}^{\PP^2/\mathfrak{C}}=1$. 
Consequently, we have $GW_{0,2F}^{\mathbb{F}_4/\mathfrak{C}_{-4}, \psi}=\frac{1}{2}$, and so Equation \eqref{eq_gw_2_2} can be rewritten as
\[\sum_{g \geq 0} GW_{g,H}^{\PP^2, \psi} u^{2g-2}=  \sum_{g \geq 0} GW_{g,H, ((2),\emptyset)}^{\PP^2/\mathfrak{C}} u^{2g-2} \,.\]
Therefore, Lemma \ref{lem_gw_2} follows from Equation \eqref{eq_gw_2}.
\end{proof}

\begin{figure}[htb]
\begin{center}
\includegraphics[height=1.2in,width=5in,angle=0]{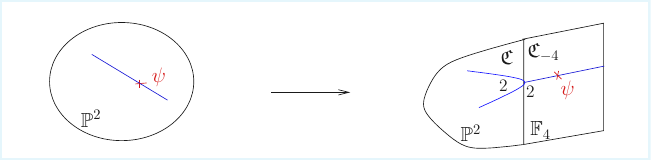}
\caption{The degeneration used in the proof of Lemma \ref{lem_gw_2}.}
\label{Fig5}
\end{center}
\end{figure}

%\begin{figure}[hbt!]
%\center{\scalebox{.7}{\input{Fig5.pspdftex}}}
%\caption{Degeneration used in the proof of Lemma \ref{lem_gw_2}.}
%\label{Fig5}
%\end{figure} 

\subsubsection{Relative BPS polynomials from refined counts of floor diagrams}
\label{subsec_relative_floor}

In this section, we prove that the relative Gromov--Witten invariants $GW_{g,\beta, (\mu,\nu)}^{\widetilde{S}_n/\widetilde{\mathfrak{C}}}$ and the 
corresponding BPS polynomials $BPS_{\beta, (\mu,\nu)}^{\widetilde{S}_n/\widetilde{\mathfrak{C}}}(q)$ defined \S\ref{sec:relative_bps} can be computed in terms of marked floor diagrams as in Definition \ref{def: marking on floor}, counted with refined multiplicities defined as follows.
Recall that every nonnegative integer $N$, we denote by $[N]_q$ the corresponding $q$-integer defined by \eqref{eq_q_integer}.

\begin{definition} \label{def_q_mult}
The \emph{refined multiplicity} of a marked floor diagram $\Gamma$ of type $(\mu,\nu)$ is
\[ m_\Gamma(q):= 
\left( \prod_{j=1}^{\ell(\nu)} \nu_j\right) 
\prod_{e\in E(\Gamma)} [w_e]_q^2 \in \ZZ[q^\pm]\,.\]
\end{definition}

\begin{remark}
    In the limit $q \rightarrow 1$, the refined multiplicity of a marked floor diagram reduces to the multiplicity $\prod_{j=1}^{\ell(\nu)} \nu_j\prod_{e\in E(\Gamma)} w_e^2$ considered in \cite[Definition 3.5]{BrugFloorconic}
\end{remark}

\begin{definition} \label{def_count_floor}
The \emph{refined count with multiplicity} of marked floor diagrams of genus zero, class $\beta$, and type $(\mu,\nu)$
\[ N_{\beta, (\mu,\nu)}^{\mathrm{floor}}(q)=\sum_\Gamma m_\Gamma(q) \in \ZZ[q^\pm]\]
where the sum is over the isomorphism classes of marked floor diagrams of class $\beta$ and type $(\mu,\nu)$.
\end{definition}

\begin{theorem}\label{thm_floor_gw}
Let $\beta \in H_2(\widetilde{S}_n,\ZZ)$ such that $H \cdot \beta \geq 1$, and $\mu$, $\nu$ as in \eqref{eq_mu_nu}. Then, the relative Gromov--Witten invariants $GW_{g,\beta, (\mu,\nu)}^{\widetilde{S}_n/\widetilde{\mathfrak{C}}}$ satisfy
\begin{align*}
&\sum_{g \geq 0}GW_{g,\beta, (\mu,\nu)}^{\widetilde{S}_n/\widetilde{\mathfrak{C}}} u^{2g-2+H \cdot \beta+\ell(\mu)+\ell(\nu)} \\
&= 
\left(\prod_{j=1}^{\ell(\mu)} \frac{[\mu_j]_q}{\mu_j} \right) 
\left(\prod_{j=1}^{\ell(\nu)} \frac{[\nu_j]_q}{\nu_j} \right)
((-i)(q^{\frac{1}{2}}-q^{-\frac{1}{2}}))^{-2+H\cdot \beta+\ell(\mu)+\ell(\nu)} N_{\beta, (\mu,\nu)}^{\mathrm{floor}}(q) \,,
\end{align*}
where $q=e^{iu}$ and $N_{\beta, (\mu,\nu)}^{\mathrm{floor}}(q)$ is the refine
count with multiplicity of marked floor diagrams 
of genus zero, class $\beta$, and type $(\mu,\nu)$.
\end{theorem}

\begin{proof}
We calculate the relative Gromov--Witten invariants
$GW_{g,\beta, (\mu,\nu)}^{\widetilde{S}_n/\widetilde{\mathfrak{C}}}$ using the degeneration
$\widetilde{\epsilon}: \widetilde{\mathcal{F}} \rightarrow \CC$ constructed in 
\S\ref{subsub: a degeneration}, with central fiber $\widetilde{\epsilon}^{-1}(0)=\PP^2 \cup \mathbb{F}_4^{(m)} \cup \cdots \cup \mathbb{F}_4^{(1)} \cup \mathrm{Bl}_n \mathbb{F}_4$.
To do this, we apply the general degeneration formalism in log Gromov--Witten theory. We endow $\CC$ with the divisorial log structure defined by the divisor $\{0\}\subset \CC$, and $\widetilde{\mathcal{F}} \rightarrow \CC$ with the divisorial log structure defined by the normal crossing divisor $\overline{\widetilde{\mathfrak{C}}\times \CC^\star} \cup \widetilde{\epsilon}^{-1}(0)$, where 
$\overline{\widetilde{\mathfrak{C}}\times \CC^\star}$ denote the closure of $\widetilde{\mathfrak{C}}\times \CC^\star$ in $\widetilde{\mathcal{F}}$. Then, $\widetilde{\epsilon}: \widetilde{\mathcal{F}} \rightarrow \CC$ is a log smooth morphism. 

By the decomposition formula of \cite{ACGS_decomposition}, the relative Gromov--Witten invariants $GW_{g,\beta}^{\widetilde{S}_n/\widetilde{\mathfrak{C}}}(\mu,\nu)$ can be expressed as a sum
of $h$-marked log Gromov--Witten invariants of the central fiber $\widetilde{\epsilon}^{-1}(0)$ endowed with the restricted log structure, where $h$ are rigid tropical maps to the tropicalization of $\widetilde{\epsilon}^{-1}(0)$. Using the vanishing properties of lambda class, we show as in \cite[\S 5.3]{bousseau_floor} that the only contributing rigid tropical maps are in natural one-to-one correspondence with the marked floor diagrams $\Gamma$ of genus zero, class $\beta$, and type $(\mu,\nu)$. The contribution of each floor diagrams can then be decomposed as a product of vertex and edge contributions as in \cite[\S 5.3]{bousseau_floor}. We refer to \S \ref{subsec: marked floor} for the notation $V(\Gamma)=V_2(\Gamma)\amalg V_4(\Gamma)$, $V_2(\Gamma)=V_2^{(2)}(\Gamma) \amalg V_2^{(1,1)}(\Gamma)$, $E(\Gamma)$, and $L(\Gamma) = L_\mu \amalg L_\nu \amalg A_1 \amalg \ldots \amalg A_n$ describing the sets of vertices, edges, and legs of a marked floor diagram $\Gamma$.

Explicitly, the contribution of a marked floor diagram $\Gamma$ to the generating series 
\[\sum_{g \geq 0}GW_{g,\beta, (\mu,\nu)}^{\widetilde{S}_n/\widetilde{\mathfrak{C}}} u^{2g-2+\beta \cdot \widetilde{\mathfrak{C}}}\] 
is the product of the following factors:
\begin{itemize}
    \item[i)] A factor $w_e^2$ for each edge $e \in E(\Gamma)$.
    \item[ii)] A factor $w_l$ for each leg $l \in L_\nu$.
    \item[iii)] By Lemma \ref{lem_gw_11}, a factor \[\sum_{g \geq 0} GW_{g,H, ((1,1),\emptyset)}^{\PP^2/\mathfrak{C}} u^{2g}=u^{-1 }2 \sin \left( \frac{u}{2}\right) =u^{-1} (-i)(q^{\frac{1}{2}}-q^{-\frac{1}{2}}),\] for each vertex $v \in V_2^{(1,1)}(\Gamma)$.
    \item[iv)] By Lemma \ref{lem_gw_2}, a factor
    \[ \sum_{g \geq 0} GW_{g,H, ((2),\emptyset)}^{\PP^2/\mathfrak{C}} u^{2g-1}=u^{-1}\cos \left( \frac{u}{2}\right) =u^{-1} \frac{1}{2}(q^{\frac{1}{2}}+q^{-\frac{1}{2}})=u^{-1}\frac{[2]_q}{2}\,\] 
    for each vertex $v\in V_2^{(2)}(\Gamma)$
\end{itemize}

\begin{itemize}
\item[v)] A factor \[ \sum_{g \geq 0} GW_{g, \mathfrak{C}_{4}+|m_v|F, (m_v, n_v)}^{\mathbb{F}_4/\mathfrak{C}_{-4}\cup \mathfrak{C}_4} u^{2g-2+\ell(m_v)+\ell(n_v)}\] for each vertex $v \in V_4(\Gamma)$, where
$m_v=(m_{v,j})_{1\leq j \leq \ell(m_v)}$ is the partition formed by the edge weights adjacent to $v$ oriented away from $v$, $|m_v|=\sum_{j=1}^{\ell(m_v)} m_{v,j}$, and
$n_v=(n_{v,j})_{1\leq j \leq \ell(n_v)}$ is the partition formed by the edge or leg weights adjacent to $v$ oriented towards $v$. 
Moreover,
\[ GW_{g, \mathfrak{C}_{4}+|m_v|F, (m_v, n_v)}^{\mathbb{F}_4/\mathfrak{C}_{-4}\cup \mathfrak{C}_4}\]
is the genus $g$ Gromov--Witten invariant of $\mathbb{F}_4$, relative to $\mathfrak{C}_{-4}\cup \mathfrak{C}_4$, of class $\mathfrak{C}_{4}+|m_v|F$, with the following insertions -- see Figure \ref{Fig10}: $(-1)^g \lambda_g$, one point constraint in $\mathbb{F}_4$ away from
$\mathfrak{C}_{-4}\cup \mathfrak{C}_4$, $\ell(m_v)$ point constraints at marked points with contact orders $(m_v)_{1\leq l \leq\ell(m_v)}$ along $\mathfrak{C}_{-4}$, and $\ell(n_v)$ point constraints at marked points with contact orders $(n_v)_{1\leq l \leq\ell(n_v)}$ along $\mathfrak{C}_{4}$.
By \cite[Theorem 4.4]{bousseau_floor}, this factor equals 
\[ u^{-2} \left((-i)(q^{\frac{1}{2}}-q^{-\frac{1}{2}})\right)^{\ell(m_v)+\ell(n_v)} \prod_{j=1}^{\ell(m_v)} \frac{[m_{v,j}]_q}{m_{v,j}} \prod_{j=1}^{\ell(n_v)} \frac{[n_{v,j}]_q}{n_{v,j}} \,.\]
\item[vi)] By \cite[Lemma 6.3]{BP_local_curves}, a factor 
\[\frac{1}{2 \sin \left(\frac{u}{2}\right)}=((-i)(q^{\frac{1}{2}}-q^{-\frac{1}{2}}))^{-1}\] 
for each leg $l \in \amalg_{i=1}^n A_i$ corresponding to maps to the $(-1)$-curve in $\mathrm{Bl}_n(\mathbb{F}_4)$ intersecting the exceptional curve $E_i$.
\end{itemize}

Consequently, the contribution of the marked floor diagram $\Gamma$ is given by
\[ \left( \prod_{e\in E(\Gamma)} w_e^2 \prod_{l \in L_\nu}w_l\right)
u^{-|V_2(\Gamma)|-2|V_4(\Gamma)|} 
((-i)(q^{\frac{1}{2}}-q^{-\frac{1}{2}}))^{|V_2^{(1,1)}(\Gamma)|+\sum_{v\in V_4(\Gamma)}(\ell(m_v)+\ell(n_v)) -\sum_{i=1}^n|A_i|} \]
\[ \times \left(\frac{[2]_q}{2}\right)^{|V_2^{(2)}(\Gamma)|}
\prod_{j=1}^{\ell(m_v)} \frac{[m_{v,j}]_q}{m_{v,j}} \prod_{j=1}^{\ell(n_v)} \frac{[n_{v,j}]_q}{n_{v,j}} \,.\]
Every edge $e$ of $\Gamma$ with weight $w_e\neq 1$ is adjacent to either two vertices in $V_4(\Gamma)$ or to one vertex in $V_4(\Gamma)$ and one vertex in $V_2^{(2)}(\Gamma)$. Similarly, every leg $l$ of $\Gamma$ with $w_l \neq 1$ is adjacent to either a vertex $v \in V_4(\Gamma)$ or to a vertex $v\in V_2^{(2)}(\Gamma)$. 
Thus, we obtain
\begin{align*}
   & \left( \prod_{e\in E(\Gamma)} w_e^2 \prod_{l \in L_\nu}w_l\right)\left(\frac{[2]_q}{2}\right)^{|V_2^{(2)}(\Gamma)|}
\prod_{j=1}^{\ell(m_v)} \frac{[m_{v,j}]_q}{m_{v,j}} \prod_{j=1}^{\ell(n_v)} \frac{[n_{v,j}]_q}{n_{v,j}}
\\=& \left(\prod_{j=1}^{\ell(\mu)} \frac{[\mu_j]_q}{\mu_j} \right) 
\left(\prod_{j=1}^{\ell(\nu)} [\nu_j]_q \right) \prod_{e\in E(\Gamma)} [w_e]_q^2\\=& m_\Gamma(q) \,.
\end{align*}
By Lemma \ref{lem_floor}, we have $|V_2(\Gamma)|+2|V_4(\Gamma)|=H \cdot \beta$, and so the 
contribution of
$\Gamma$ simplifies to 
\[ u^{-H \cdot \beta}\left(\prod_{j=1}^{\ell(\mu)} \frac{[\mu_j]_q}{\mu_j} \right) 
\left(\prod_{j=1}^{\ell(\nu)} \frac{[\nu_j]_q}{\nu_j} \right)((-i)(q^{\frac{1}{2}}-q^{-\frac{1}{2}}))^{|V_2^{(1,1)}(\Gamma)|+\sum_{v\in V_4(\Gamma)}(\ell(m_v)+\ell(n_v)) -\sum_{i=1}^n|A_i|} m_\Gamma(q)\,,\]
where we used the Definition \ref{def_q_mult} of $m_\Gamma(q)$.

Therefore, to prove Theorem \ref{thm_floor_gw}, it remains only to show that 
\begin{equation}\label{eq_combinatorics}
|V_2^{(1,1)}(\Gamma)|+\sum_{v\in V_4(\Gamma)}(\ell(m_v)+\ell(n_v)) -\sum_{i=1}^n|A_i|= -2+H \cdot \beta +\ell(\mu)+\ell(\nu)\,.
\end{equation}
We first observe that $\sum_{v\in V_4(\Gamma)}(\ell(m_v)+\ell(n_v))$ equals the total number of half-edges or legs adjacent to vertices in $V_4(\Gamma)$, that is, all half-edges of legs except the ones adjacent to vertices in $V_2(\Gamma)$. Hence, we have 
\[ \sum_{v\in V_4(\Gamma)}(\ell(m_v)+\ell(n_v)) = 2|E(\Gamma)|+|L(\Gamma)|-2 |V_2^{(1,1)}(\Gamma)|-|V_2^{(2)}(\Gamma)|\,.\]
Thus, we have 
\begin{align*}
   & |V_2^{(1,1)}(\Gamma)|+\sum_{v\in V_4(\Gamma)}(\ell(m_v)+\ell(n_v)) -\sum_{i=1}^n|A_i| \\
   & = 2|E(\Gamma)|+|L(\Gamma)|-|V_2^{(1,1)}(\Gamma)|-|V_2^{(2)}(\Gamma)|-\sum_{i=1}^n|A_i|
    \\&=2|E(\Gamma)|
+|L(\Gamma)|-|V_2(\Gamma)|-\sum_{i=1}^n|A_i|\,.
\end{align*}
Since the graph $\Gamma$ is of genus zero, we have $|E(\Gamma)|=|V(\Gamma)|-1$, so we obtain
\[|V_2^{(1,1)}(\Gamma)|+\sum_{v\in V_4(\Gamma)}(\ell(m_v)+\ell(n_v))-\sum_{i=1}^n|A_i|=2|V(\Gamma)|-2+|L(\Gamma)|-|V_2(\Gamma)|-\sum_{i=1}^n|A_i|\,.\]
Using that  $|V(\Gamma)|=|V_2(\Gamma)|+|V_4(\Gamma)|$, 
$|L(\Gamma)|-\sum_{i=1}^n|A_i|=\ell(\mu)+\ell(\nu)$, and $|V_2(\Gamma)|+2|V_4(\Gamma)|=H \cdot \beta$ by Lemma \ref{lem_floor}, we finally
obtain Equation \eqref{eq_combinatorics}, and this concludes the proof of Theorem \ref{thm_floor_gw}.
\end{proof}

\begin{figure}[htb]
\begin{center}
\includegraphics[height=0.8in,width=2in,angle=0]{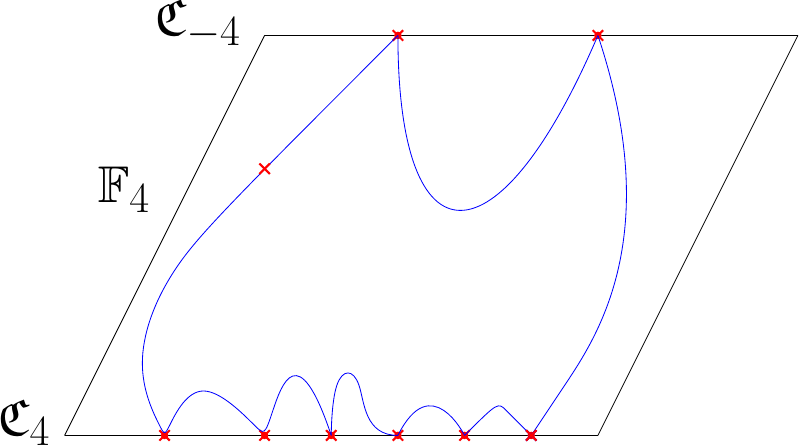}
 \caption{Curve contributing to  $GW_{g, \mathfrak{C}_{4}+|m_v|F, (m_v, n_v)}^{\mathbb{F}_4/\mathfrak{C}_{-4}\cup \mathfrak{C}_4}$}
    \label{Fig10}
\end{center}
\end{figure}

% \begin{figure}
 %   \center{\includegraphics[height=1in,width=2in,angle=0]{Fig10.eps}}
  %  \caption{Curve contributing to  $GW_{g, \mathfrak{C}_{4}+|m_v|F, (m_v, n_v)}^{\mathbb{F}_4/\mathfrak{C}_{-4}\cup \mathfrak{C}_4}$}
   % \label{Fig10}
   % \end{figure}

\begin{corollary} \label{cor_bps_floor}
    Let $\beta \in H_2(\widetilde{S}_n,\ZZ)$ such that $H \cdot \beta \geq 1$, and $\mu$, $\nu$ as in \eqref{eq_mu_nu}.  Then, the relative BPS polynomial $BPS_{\beta,(\mu,\nu)}^{\widetilde{S}_n/\mathfrak{C}}(q)$ 
    is equal to the refined count with multiplicity $N_{\beta, (\mu,\nu)}^{\mathrm{floor}}(q)$ of marked floor diagrams:
\[ BPS_{\beta}^{\widetilde{S}_n/\widetilde{\mathfrak{C}}}(q) =N_{\beta, (\mu,\nu)}^{\mathrm{floor}}(q) \,. \]
In particular, $BPS_{g,\beta}^{\widetilde{S}_n/\widetilde{\mathfrak{C}}}(q)$ is a Laurent polynomial with integer coefficients:
\[ BPS_{g,\beta}^{\widetilde{S}_n/\widetilde{\mathfrak{C}}}(q) \in \ZZ[q^\pm] \,.\]
\end{corollary}

\begin{proof}
The result follows from comparing Theorem \ref{thm_floor_gw} with Definitions \ref{def_relative_bps_inv}-\ref{def_relative_BPS}, which define the relative BPS invariants and relative BPS polynomials.
\end{proof}

\subsection{Relative BPS polynomials and relative Welschinger counts}
\label{sec_relative_bps_w}

In this section, we first define relative Welschinger counts of real rational curves in $(\widetilde{S}_n, \widetilde{\mathfrak{C}})$. 
Then, we prove in Theorem \ref{thm: relative_BPS_W} that the specialization at $q=-1$ of the relative BPS polynomials of $(\widetilde{S}_n,\widetilde{\mathfrak{C}})$ can be expressed in terms of these relative Welschinger counts.

\subsubsection{Relative Welschinger counts}
\label{subsec:relative Welschinger counts}
In this section, we equip $\PP^2$ with its standard real structure whose real locus is $\RR\PP^2$, and we fix $\mathfrak{C}$ a smooth real conic in $\PP^2$ with non-empty real locus.
We denote by $\widetilde{S}_n$ the real surface obtained by blowing up $\PP^2$ at $n$ general real points of $\mathfrak{C}$, and we denote by 
$\widetilde{\mathfrak{C}}$ the strict transform of $\mathfrak{C}$ in $\widetilde{S}_n$.

\begin{definition}
\label{def: real points}
Let $\beta \in H_2(\widetilde{S}_n,\ZZ)$ such that $H \cdot \beta \geq 1$, and $\mu$, $\nu$ as in \eqref{eq_mu_nu}. 
A \emph{configuration of real points} of type $(\beta,\mu,\nu)$ is given by a disjoint union 
\[   \mathbf{x} =  \mathbf{x}^0 \amalg  \mathbf{x}^{\widetilde{\mathfrak{C}}} \,,\]
where $ \mathbf{x}^0=(x_1^0,\ldots, x^0_{m_{\beta,(\mu,\nu)}})$ is a set of $m_{\beta,(\mu,\nu)}:=H \cdot \beta -1+\ell(\nu)$ real points in $\widetilde{S}_n \setminus \widetilde{\mathfrak{C}}$, and $\mathbf{x}^{\widetilde{\mathfrak{C}}} = (x_1^{\widetilde{\mathfrak{C}}},\ldots, x_{\ell(\nu)}^{\widetilde{\mathfrak{C}}} )$ is a set of $\ell(\mu)$ real points on $\widetilde{\mathfrak{C}}$.
\end{definition}

Given a configuration of real points $\mathbf{x}$ of type $(\beta,\mu,\nu)$, we consider the set $M^{\RR}_{0,\beta}(\widetilde{S}_n/\widetilde{\mathfrak{C}}, \mu,\nu,\mathbf{x})$ of real genus zero stable maps 
\begin{equation}
\label{eq: f}
f: (C,(p_i^0)_{1\leq i \leq m_{\beta,(\mu,\nu)}} , (p_j^{\widetilde{\mathfrak{C}}})_{1\leq j \leq \ell(\mu)} , (p_k)_{1\leq k \leq \ell(\nu)} ) \to \widetilde{S}_n  \,,    
\end{equation}
of class $\beta$ satisfying the following:
\begin{itemize}
    \item[1)] $f(p_i^0)  = x_i^0$
\item[2)] $f(p_j^{\widetilde{\mathfrak{C}}})  = x_j^{\widetilde{\mathfrak{C}}} ~ \text{and the contact order of} ~ f(p_j^{\widetilde{\mathfrak{C}}}) ~ \text{with} ~ \widetilde{\mathfrak{C}} ~ \text{is} ~ \mu_j,$
\item[3)] $f(p_k)$ has contact order $\nu_k$ with $\widetilde{\mathfrak{C}}$,
\item[4)] $\widetilde{\mathfrak{C}}$ is not a component of $f(C)$: In particular, $f(C)$ intersects $\widetilde{\mathfrak{C}}$ only at the union of points $f(p_j^{\widetilde{\mathfrak{C}}})$ for $1\leq j \leq \ell(\mu)$ and $f(p_k)$ for $1\leq k \leq \ell(\nu)$.
\end{itemize}
If $\mathbf{x}$ is a generic configuration of real points of type $(\beta,\mu,\nu)$, then, by \cite[Proposition 2.1]{SS}, then the domain curve $C$ of such a stable map is smooth and irreducible, hence $C \cong \PP^1$. Moreover, $f$ is an immersion, birational onto its image, $f(C)$ intersects $\widetilde{\mathfrak{C}}$ only at non-singular points, and the set $M^{\RR}_{0,\beta}(\widetilde{S}_n/\widetilde{\mathfrak{C}}, \mu,\nu,\mathbf{x})$ is finite. Hence, as in \cite[\S2.3]{BrugFloorconic}, we can define a relative Welschinger count
$W_{\beta, (\mu,\nu), \mathbf{x}}^{\widetilde{S}_n/\widetilde{\mathfrak{C}}}$
by 
\[W_{\beta,(\mu,\nu),\mathbf{x}}^{\widetilde{S}_n/\widetilde{\mathfrak{C}}} 
= \sum_{f \in M^{\RR}_{0,\beta}(\widetilde{S}_n/\widetilde{\mathfrak{C}}, \mu,\nu,\mathbf{x})} w(f)\,, \]
Here, $w(f) \in \{\pm 1\}$ is the Weslchinger sign
defined as in \cite[\S 3.6]{IKS2} by  
\[w(f)=(-1)^{s(f)}\,,\]
where
\begin{equation}
s(f)= \sum_{z \in \text{Sing(f(C))}}     s(f,z) \,,
\end{equation}
where the sum runs over all points $z$ in the singular locus $\text{Sing}(f(C))$ of $f(C)$, and $s(f,z)$ denotes the number of pairs of imaginary complex conjugate local branches of $f(C)$ at $z$, each pair being counted with the weight equal to the intersection number of the branches.

As noted in \cite[\S 2.3]{BrugFloorconic}, the relative Welschinger counts $W_{\beta, (\mu,\nu), \mathbf{x}}^{\widetilde{S}_n/\widetilde{\mathfrak{C}}}$ may vary with the choice of $\mathbf{x}$.

\subsubsection{Relative Welschinger counts from relative BPS polynomials at $q=-1$}

In the following result, we relate the specialization at $q=-1$ of the relative BPS polynomial $BPS_{\beta,(\mu,\nu)}^{\widetilde{S}_n/\widetilde{\mathfrak{C}}}(q)$ with the relative Welschinger counts 
$W_{\beta, (\mu,\nu),\mathbf{x}}^{\widetilde{S}_n/\widetilde{\mathfrak{C}}}$.
We use the following notation: for every positive integer $k$, we set $[k]_\RR=1$ if $k$ is odd, and $[k]_\RR=2$ if $k$ is even.

\begin{theorem}
\label{thm: relative_BPS_W}
 Let $\beta \in H_2(\widetilde{S}_n,\ZZ)$ such that $H \cdot \beta \geq 1$, and $\mu$, $\nu$ as in \eqref{eq_mu_nu}.  
 There exists a generic configuration of real points $\mathbf{x}$ of type $(\beta,\mu,\nu)$, such that the specialization at $q=-1$ of the relative BPS polynomial $BPS_{\beta,(\mu,\nu)}^{\widetilde{S}_n/\widetilde{\mathfrak{C}}}(q)$ and the relative Welschinger counts of $W_{\beta, (\mu,\nu),\mathbf{x}}^{\widetilde{S}_n/\widetilde{\mathfrak{C}}}$ of real rational curves passing through $\mathbf{x}$ are related by:
    \[ BPS_{\beta, (\mu,\nu)} ^{\widetilde{S}_n/\widetilde{\mathfrak{C}}}(-1)=
    \left( \prod_{j=1}^{\ell(\nu)} \frac{v_j}{[\nu_j]_\RR}\right)
    W_{\beta, (\mu,\nu), \mathbf{x}}^{\widetilde{S}_n/\widetilde{\mathfrak{C}}} \]
\end{theorem}

\begin{proof}
By \cite[Theorem 3.12]{BrugFloorconic}, there exists a generic configuration of real points $\mathbf{x}$ of type $(\beta,\mu,\nu)$, such that the relative Welschinger counts of $W_{\beta, (\mu,\nu), \mathbf{x}}^{\widetilde{S}_n/\widetilde{\mathfrak{C}}}$ is given by 
\[ W_{\beta, (\mu,\nu), \mathbf{x}}^{\widetilde{S}_n/\widetilde{\mathfrak{C}}} = \sum_\Gamma m_{\Gamma,\RR}\,,\]
where the sum is over the marked floor diagrams of genus zero, class $\beta$, and type $(\mu,\nu)$.
The real multiplicity $m_{\Gamma,\RR}$ is determined as described in \cite[Definition 3.10]{BrugFloorconic}. 
Using the notation of \cite[Definition 3.10]{BrugFloorconic}, we only consider the case $(s,\kappa)=(0,0)$, corresponding to purely real constraints and to the standard real structure. This leads to the simple result that
$m_{\Gamma,\RR}=0$ if $\Gamma$ contains an edge of even weight, and $m_{\Gamma,\RR}=\prod_{j=1}^{\ell(\nu)} [\nu_j]_\RR$ else.

On the other hand, by Corollary \ref{cor_bps_floor}, we have $BPS_{\beta, (\mu,\nu)} ^{\widetilde{S}_n/\widetilde{\mathfrak{C}}}(-1)=N_{\beta, (\mu,\nu)}^{\mathrm{floor}}(-1)$. Applying Definition \ref{def_q_mult}, we then deduce 
\[BPS_{\beta, (\mu,\nu)} ^{\widetilde{S}_n/\widetilde{\mathfrak{C}}}(-1)= \sum_\Gamma m_\Gamma(-1) \,,\]
where $m_\Gamma(-1)=\left( \prod_{j=1}^{\ell(\nu)} \nu_j \right) \prod_{e \in E(\Gamma)} [w_e]^2_{q=-1}$.
By \eqref{eq_q_integer}, we have $[w_e]^2_{q=-1}=0$ if $w_e$ is even, and $[w_e]^2_{q=-1}=1$ if $w_e$ is odd. Therefore, for every marked floor diagram $\Gamma$, we obtain
\[ m_\Gamma(-1)= \left( \prod_{j=1}^{\ell(\nu)} \frac{v_j}{[\nu_j]_\RR}\right) m_{\Gamma, \RR}\,,\]
and so \[ BPS_{\beta, (\mu,\nu)} ^{\widetilde{S}_n/\widetilde{\mathfrak{C}}}(-1)=
  \left( \prod_{j=1}^{\ell(\nu)} \frac{v_j}{[\nu_j]_\RR}\right)
    W_{\beta, (\mu,\nu),\mathbf{x}}^{\widetilde{S}_n/\widetilde{\mathfrak{C}}}\,. \]
\end{proof}

\begin{figure}[hbt!]
\center{\includegraphics[height=3in,width=6in,angle=0]{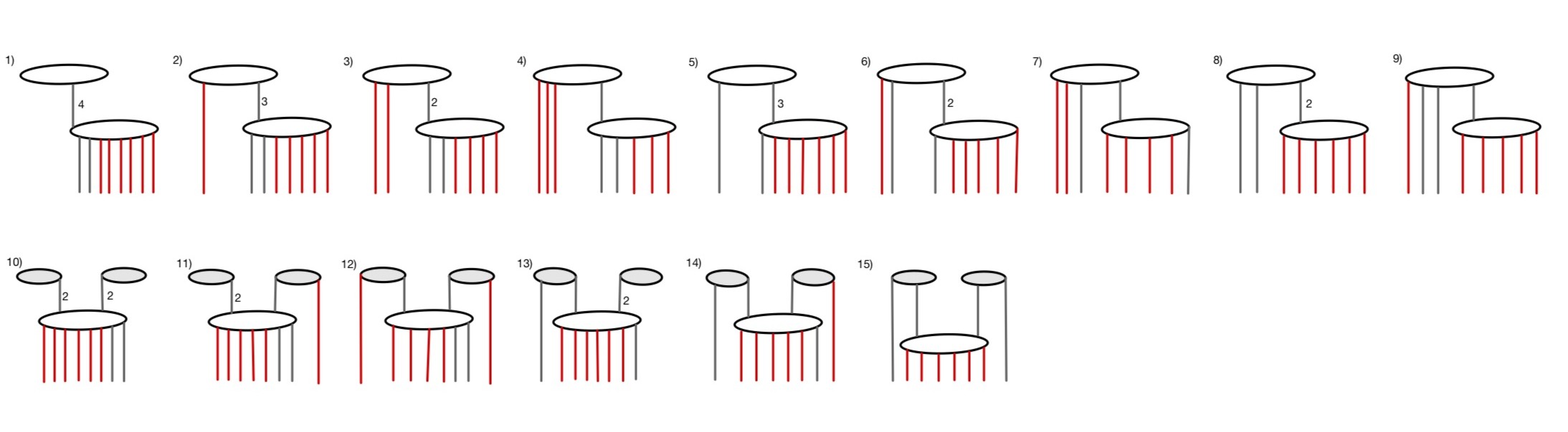}}
\caption{All floor diagrams of degree $4$ of class $\beta = 4H  - \sum_{i=1}^6 E_i \in H_2(\widetilde{S}_6,\ZZ)$ and type $(\emptyset, (1,1))$.}
\label{Fig7}
\end{figure}

\begin{table}[]
    \centering
    \begin{tabular}{l|l|l|l} 
    \hline
Diagram & $\CC$-count & $\RR$-count & Refined count  \\ 
\hline
$1)$ & $16$ & $0$ & $q^{-3} + 2 q^{-2} + 3 q^{-1} + 4 +3q + 2q^2 +q^3 $ \\
\hline
$2)$ & $54$ & $6$ & $6 q^{-2} + 12 q^{-1} + 18 +12 q + 6q^2 $ \\
\hline
$3)$ & $60$ & $0$ & $ 15 q^{-1} + 30 + 15 q $ \\
\hline
$4)$ & $20$ & $20$ & $20 $ \\
\hline
$5)$ & $36$ & $4$ & $4 q^{-2} + 8 q^{-1} + 12 +8 q + 4q^2 $ \\
\hline
$6)$ & $96$ & $0$ & $ 24 q^{-1} + 48 + 24 q $ \\
\hline
$7)$ & $60$ & $60$ & $ 60 $ \\
\hline
$8)$ & $24$ & $0$ & $ 6 q^{-1} + 12 + 6 q$ \\
\hline
$9)$ & $36$ & $36$ & $ 36$ \\
\hline
$10)$ & $16$ & $0$ & $ 6 q^{-1} + 12 + 6 q$ \\
\hline
$11)$ & $48$ & $0$ & $q^{-2} + 4 q^{-1} + 6 + 4 q + q^2$ \\
\hline
$12)$ & $30$ & $30$ & $ 30 $ \\
\hline
$13)$ & $40$ & $0$ & $ 10 q^{-1} + 20 + 10 q$ \\
\hline
$14)$ & $60$ & $60$ & $ 60 $ \\
\hline
$15)$ & $20$ & $20$ & $ 20 $ \\
\hline
  \end{tabular}
    \caption{The complex, real, and refined counts of floor diagrams in Figure \ref{Fig7}.}
    \label{Table1}
\end{table}

\subsection{Examples of relative BPS polynomials}
\label{sec_examples}

\begin{example} \label{ex_relative_4}
    Let $n=6$,  $\beta = 4H - \sum_{i=1}^6 E_i \in H_2(\widetilde{S}_6,\ZZ)$, and $(\mu, \nu)=(\emptyset, (1,1))$, so that $m_{\beta,(\mu,\nu)}=4-1+2=5$.
    We represent in Figures \ref{Fig7} all the genus zero floor diagrams of class $\beta$ and type $(\mu,\nu)$ -- these floor diagrams can also be found in \cite[Figure 3]{BrugFloorconic}, where the red legs are omitted. We list the complex, real, and refined contributions of each floor diagram in Table \ref{Table1}. In particular, summing all the refined counts, we obtain by Corollary \ref{cor_bps_floor} the corresponding BPS polynomial 
    \[BPS_{\beta, (\mu,\nu)}^{\widetilde{S}_6/\widetilde{\mathfrak{C}}}(q)=q^{-3}+13q^{-2}+94q^{-1}+400+94q+13q^2+q^3\,,\]
    interpolating between the complex count $BPS_{\beta, (\mu,\nu)} ^{\widetilde{S}_6/\widetilde{\mathfrak{C}}}(1) = GW_{0,\beta, (\mu,\nu)} ^{\widetilde{S}_6/\widetilde{\mathfrak{C}}}=616$ discussed in \cite[Example 3.8]{BrugFloorconic}, and the real count $BPS_{\beta, (\mu,\nu)} ^{\widetilde{S}_6/\widetilde{\mathfrak{C}}}(-1) = W_{\beta, (\mu,\nu),\mathbf{x}} ^{\widetilde{S}_6/\widetilde{\mathfrak{C}}}=236$ appearing in \cite[Table 1]{BrugFloorconic}.
\end{example}

\begin{example} \label{ex_relative_6}
Let $n=6$,  $\beta = 6H - 2\sum_{i=1}^6 E_i \in H_2(\widetilde{S}_6,\ZZ)$, and $(\mu, \nu)=(\emptyset, \emptyset)$, so that $m_{\beta,(\mu,\nu)}=6-1=5$.
We represent in Figures \ref{Fig7} all the genus zero floor diagrams of class $\beta$ and type $(\mu,\nu)$ -- these floor diagrams can also be found in 
\cite[Figures 4-5]{BrugFloorconic}, where the red legs are omitted. We list the complex, real, and refined contributions of each floor diagram in Table \ref{Table1}. In particular, summing all the refined counts, we obtain by Corollary \ref{cor_bps_floor} the corresponding BPS polynomial 
    \[BPS_{\beta, (\mu,\nu)} ^{\widetilde{S}_6/\widetilde{\mathfrak{C}}}(q)=q^{-4}+11q^{-3}+74q^{-2}+359q^{-1}+1112+359q
+74q^2+11q^3+q^4\,,\]
    interpolating between the complex count $BPS_{\beta, (\mu,\nu)} ^{\widetilde{S}_6/\widetilde{\mathfrak{C}}}(1) = GW_{0,\beta, (\mu,\nu)} ^{\widetilde{S}_6/\widetilde{\mathfrak{C}}}=2002$ discussed in \cite[Example 3.8]{BrugFloorconic}, and the real count $BPS_{\beta, (\mu,\nu)} ^{\widetilde{S}_6/\widetilde{\mathfrak{C}}}(-1) = W_{\beta, (\mu,\nu),\mathbf{x}} ^{\widetilde{S}_6/\widetilde{\mathfrak{C}}}=522$ appearing in \cite[Table 2]{BrugFloorconic}.
\end{example}

\begin{figure}[htb]
\begin{center}
\includegraphics[height=4in,width=6in,angle=0]{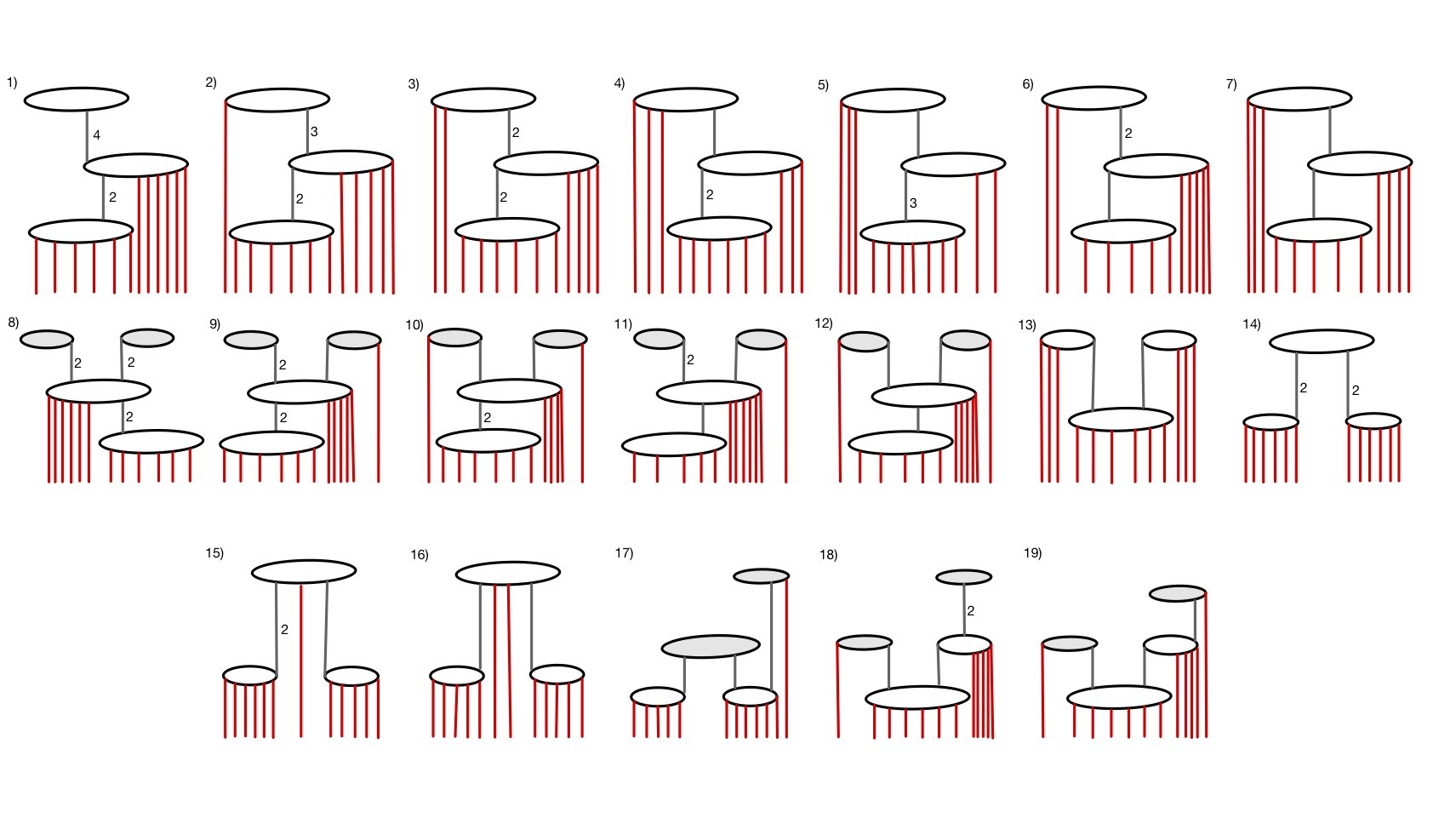}
\caption{All floor diagrams of degree $6$ of class $\beta = 6 H - 2\sum_{i=1}^6 E_i \in H_2(\widetilde{S}_6,\ZZ)$ and type $(\emptyset, \emptyset)$.}
\label{Fig8}
\end{center}
\end{figure}

\begin{table}[]
    \centering
    \begin{tabular}{l|l|l|l} 
    \hline
Diagram & $\CC$-count & $\RR$-count & Refined count  \\ 
\hline
$1)$ & $64$ & $0$ & $q^{-4} + 4 q^{-3} + 8 q^{-2} + 12 q^{-1} + 14 +12q + 8q^2 + 4 q^3 + q^4 $ \\ \hline
$2)$ & $216$ & $0$ & $6 q^{-3} + 24 q^{-2} + 48 q^{-1} + 60 +48 q + 24 q^2 + 6q^3 $ \\
\hline
$3)$ & $240$ & $0$ & $ 15 q^{-2} + 60 q^{-1} + 90 + 60 q + 15 q^2 $ \\
\hline
$4)$ & $80$ & $20$ & $20q^{-1} + 40 + 20q  $ \\
\hline
$5)$ & $54$ & $6$ & $6 q^{-2} + 12 q^{-1} + 18 +12 q + 6q^2 $ \\
\hline
$6)$ & $120$ & $0$ & $ 30 q^{-1} + 60 + 30 q $ \\
\hline
$7)$ & $60$ & $60$ & $ 60 $ \\
\hline
$8)$ & $64$ & $0$ & $q^{-3} + 6q^{-2} + 15 q^{-1} + 20 + 15 q + 6 q^2 + q^3$ \\
\hline
$9)$ & $192$ & $0$ & $ 12q^{-2} + 48 q^{-1} + 72 + 48 q + 12 q^2$ \\
\hline
$10)$ & $120$ & $0$ & $ 30 q^{-1} + 60 + 30 q$ \\
\hline
$11)$ & $48$ & $0$ & $ 12 q^{-1} + 24 + 12 q $ \\
\hline
$12)$ & $66$ & $66$ & $ 66 $ \\
\hline
$13)$ & $60$ & $60$ & $ 60 $ \\
\hline
$14)$ & $48$ & $0$ & $ 3q^{-2} + 12 q^{-1} + 18 + 12q + 3 q^2 $ \\
\hline
$15)$ & $144$ & $0$ & $ 36q^{-1} + 72 + 36q $ \\
\hline
$16)$ & $90$ & $90$ & $ 90 $ \\
\hline
$17)$ & $120$ & $120$ & $ 120 $ \\
\hline
$18)$ & $96$ & $0$ & $ 24q^{-1} + 48 + 24 q $ \\
\hline
$19)$ & $120$ & $120$ & $ 120 $ \\
\hline
  \end{tabular}
    \caption{The complex, real, and refined counts of floor diagrams in Figure \ref{Fig8}.}
    \label{Table2}
\end{table}

\section{BPS polynomials and Welschinger invariants of del Pezzo surfaces}
\label{section_bps_welschinger}

In \S \ref{sec_ABV}, we prove a refined version of the 
Abramovich--Bertram--Vakil formula which relates the BPS polynomials of the cubic surface $S_6$ with the relative BPS polynomials of $(\widetilde{S}_6, \widetilde{\mathfrak{C}})$. We use this result to prove Theorem \ref{thm_main} in 
\S \ref{sec_thm_main}, showing that the specialization at $q=-1$ of the BPS polynomials of the surfaces $S_n$ with $n\leq 6$ coincides with the Welschinger invariants.

\subsection{Refined Abramovich--Bertram--Vakil formula}
\label{sec_ABV}

In this section, we study the BPS polynomials of the surface $S_6$ obtained from $\PP^2$ by blowing-up $6$ points in general position, that is, of a smooth cubic surface in $\PP^3$. Recall from \S\ref{sec:relative_bps} that we denote by 
$\widetilde{S}_6$ the blow-up of $\PP^2$ at 6 points lying on a smooth conic, and by $\widetilde{\mathfrak{C}}$ the strict transform of the conic, so that, $\widetilde{\mathfrak{C}}^2
=2^2-6=-2$. In \cite[\S 9.2]{vakil2000counting}, Vakil studies the Gromov--Witten invariants of $S_6$ by degeneration to $(\widetilde{S}_6, \widetilde{\mathfrak{C}}^2)$, generalizing a previous formula of Abramovich-Bertram \cite{AB_12} which relates Gromov--Witten invariants of $\mathbb{F}_0$ and $\mathbb{F}_2$. The following result is a refined version of the Abramovich--Bertram--Vakil formula for $S_6$.

\begin{theorem} \label{thm_refined_ABV}
    For every $\beta \in H_2(S_6,\ZZ)$ such that $\beta \cdot \widetilde{\mathfrak{C}} \geq 1$ and for every $g \geq 0$, we have 
    \[ GW_{g,\beta}^{S_6}
    =\sum_{k \geq 0} \binom{\beta \cdot \widetilde{\mathfrak{C}}+2k}{k}
GW^{\widetilde{S}_6/\widetilde{\mathfrak{C}}}_{g,\beta-k \widetilde{\mathfrak{C}}, (\emptyset, \nu_k)}\,,\]
    where $\nu_k$ is the partition consisting of $(\beta \cdot \widetilde{\mathfrak{C}}+2k)$ parts equal to $1$. In particular, we have 
     \[ BPS_{\beta}^{S_6}(q)
    =\sum_{k \geq 0} \binom{\beta \cdot \widetilde{\mathfrak{C}}+2k}{k}
    BPS^{\widetilde{S}_6/\widetilde{\mathfrak{C}}}_{\beta-k \widetilde{\mathfrak{C}}, (\emptyset, \nu_k)}(q)\,.\]
\end{theorem}

\begin{proof}
The refined version of the original Abramovich-Bertram formula relating Gromov--Witten invariants of $\mathbb{F}_0$ and $\mathbb{F}_2$ is proved in \cite[Theorem 8.3]{bousseau_floor}. An analogous degeneration argument applied to $S_6$ and $(\widetilde{S}_6, \widetilde{\mathfrak{C}}^2)$ proves Theorem \ref{thm_refined_ABV}.
\end{proof}

\subsection{Welschinger invariants from BPS polynomials at $q=-1$}
\label{sec_thm_main}
Recall that we denote by $S_n$ a surface obtained from $\PP^2$ by blowing-up $n$ general real points.
Conjecture \ref{conj_bps_w} states that, for any $n$, the specialization at $q=-1$ of the BPS polynomials of $S_n$ is equal to the Welschinger invariants of $S_n$. Below we prove that this is true for $n \leq 6$.

\begin{theorem} \label{thm_main}
Let $S_n$ be a blow-up of $\PP^2$ at $n$ general real points. Then, for every $n \leq 6$ and $\beta \in H_2(S_n, \ZZ)$ such that $m_{\beta}:= -1 + c_1(S_n) \cdot \beta \geq 0$, the following holds:
     \[ BPS_{\beta}^{S_n}(-1)=W_{\beta}^{S_n} \,,\]
that is, Conjecture \ref{conj_bps_w} holds for $n \leq 6$.
\end{theorem}

\begin{proof}
By the refined Abramovich--Bertram--Vakil formula of Theorem \ref{thm_refined_ABV}, we have
\begin{equation} \label{eq_ABV}
BPS_{\beta}^{S_6}(q)
    =\sum_{k \geq 0} \binom{\beta \cdot \widetilde{\mathfrak{C}}+2k}{k}
    BPS^{\widetilde{S}_6/\widetilde{\mathfrak{C}}}_{\beta-k \widetilde{\mathfrak{C}}, (\emptyset, \nu_k)}(q)\,.\end{equation}
On the other hand, by the real Abramovich--Bertram--Vakil formula, proved symplectically in \cite[Theorem 2.2]{BP1}, \cite[Theorem 7]{BP2}, and algebraically in \cite[\S 4]{IKS2}, for every generic configuration of real points $\mathbf{x}_k$ of type $(\beta -k \widetilde{\mathfrak{C}}, \emptyset, \nu_k)$,
we have \begin{equation} \label{eq_real_ABV}
W_{\beta}^{S_6}
    =\sum_{k \geq 0} \binom{\beta \cdot \widetilde{\mathfrak{C}}+2k}{k}
    W^{\widetilde{S}_6/\widetilde{\mathfrak{C}}}_{\beta-k \widetilde{\mathfrak{C}}, (\emptyset, \nu_k), \mathbf{x}_k}\,.\end{equation}
For every $k \in\ZZ_{\geq 0}$, Theorem \ref{thm: relative_BPS_W} ensures the existence of $\mathbf{x}_k$ such that, given that all parts of $\nu_k$ are equal to one, the following holds:
\begin{equation} \label{eq_relative}BPS^{\widetilde{S}_6/\widetilde{\mathfrak{C}}}_{\beta-k \widetilde{\mathfrak{C}}, (\emptyset, \nu_k)}(-1)=W^{\widetilde{S}_6/\widetilde{\mathfrak{C}}}_{\beta-k \widetilde{\mathfrak{C}}, (\emptyset, \nu_k), \mathbf{x}_k}\,.\end{equation}
The result for $S_6$ follows by combination of \eqref{eq_ABV}, \eqref{eq_real_ABV}, and \eqref{eq_relative}. This implies the result for all $n \leq 6$ by Lemma \ref{lem_blow_up}.
\end{proof}

\begin{example} \label{example_3240_refined}
Consider $\beta=2c_1(S_6)=6H-2\sum_{i=1}^6 E_i \in H_2(S_6,\ZZ)$.
We have $\beta \cdot \widetilde{\mathfrak{C}}=0$, and so,
by Theorem \ref{thm_main}, 
we obtain 
\[ BPS_\beta^{S_6}(q)=BPS_\beta^{\widetilde{S}_6/\widetilde{\mathfrak{C}}}(q) + 
2\, BPS_{\beta'}^{\widetilde{S}_6/\widetilde{\mathfrak{C}}}(q)+6 \,, \]
where $\beta'=\beta-\widetilde{\mathfrak{C}}
=4H-\sum_{i=1}^6 E_i$.
The relative BPS polynomials $BPS_\beta^{\widetilde{S}_6/\widetilde{\mathfrak{C}}}(q)$ and $BPS_{\beta'}^{\widetilde{S}_6/\widetilde{\mathfrak{C}}}(q)$ are calculated in Examples \ref{ex_relative_6} and \ref{ex_relative_4} respectively. Using these results, we obtain
\[ BPS_\beta^S(q)=q^{-4}+13q^{-3}+100 q^{-2}+547 q^{-1}+1918+547 q+100 q^2+13 q^3+q^4 \,,\]
interpolating between $GW_{0,\beta}^{S_6}=3240$
at $q=1$ (see Example \ref{ex_3240_complex}) and $W_\beta^{S_6}=1000$ at $q=-1$ (see Example \ref{ex_3240_real}).
\end{example}

\section{BPS polynomials and K-theoretic refined BPS invariants}
\label{section_refined_dt}

Under suitable positivity assumptions, the Block-G\"ottsche polynomials of toric surfaces are conjectured to be related to polynomials defined in terms of Hirzebruch genera of relative Hilbert schemes of points on universal curves over linear systems \cite[Conjecture 6.12]{GSrefined}. In this section, we formulate a version of this conjecture in the broader context of BPS polynomials of surfaces defined in \S\ref{section_bps_surfaces}.

\subsection{BPS polynomials and K-theoretic refined BPS invariants}

As in \S\ref{section_bps_surfaces}, let $S$ be a smooth projective surface, and $\beta \in H_2(S,\ZZ)$ such that $m_\beta:=-1+c_1(S)\cdot \beta \geq 0$. 
The associated BPS polynomial $BPS_\beta^S(q)\in \ZZ[q^\pm]$ is defined in Definition \ref{Defn:BPSS} in terms of higher genus Gromov--Witten theory of the 3-fold $S \times \PP^1$, or equivalently the $\CC^\star$-equivariant higher genus Gromov--Witten theory of the 3-fold $S \times \CC$.

On the other hand, let $K_S$ be the non-compact Calabi--Yau 3-fold obtained by considering the total space of the canonical line bundle of $S$. There is a natural $\CC^\star$-action on $K_S$ scaling the fibers of the projection $K_S \rightarrow S$. The $K$-theoretic refined genus $0$ BPS invariant $\Omega_\beta^{K_S}(q) \in \ZZ[q^{\pm}]$ of $K_S$ with $m_\beta$ point insertions is given by
\[ \Omega_\beta^{K_S}(q)=\chi_q(M_{\beta,1}^{K_S}, \hat\cO^\vir \otimes \bigotimes_{i=1}^{m_\beta} \tau(p_i)) \,,\]
where $M_{\beta,1}^{K_S}$ is the moduli space of stable one-dimensional sheaves on $S$ of class $\beta$ and Euler characteristic one, $\hat\cO^\vir$ is the corresponding Nekrasov-Okounkov twisted virtual structure sheaf, and $\tau(p_i)$ are tautological classes imposing the $m_\beta$ point constraints. Finally, $\chi_q$ is the $\CC^\star$-equivariant Euler characteristic with respect to the $\CC^\star$-action on $K_S$, and we denote by $q$ the equivariant parameter. We refer to \cite{NO, Thomas_K_VW, thomas2024refined} for the details of the definition,  and in particular to \cite{afgani2020refinements} for a discussion of point constraints.
Alternatively, one could use moduli spaces of stable pairs on $K_S$, with the conjectural expectation that they yield the same invariants. 

\begin{conjecture} \label{conj_DT}
    Let $S$ be a smooth projective surface, and $\beta \in H_2(S,\ZZ)$ such that $m_\beta:=-1+c_1(S)\cdot \beta \geq 0$. Then, the BPS polynomial $BPS_\beta^S(q)$ of $S$ coincides with the refined genus $0$ BPS invariant of $K_S$ with $m_\beta$ point insertions:
    \[ BPS_\beta^S(q) = \Omega_\beta^{K_S}(q)\,. \]
\end{conjecture}

This conjecture suggests a striking connection between the Gromov--Witten theory of $S \times \CC$ and refined sheaf counting on $K_S$.
From a physics perspective, this conjecture aligns with expectations about the fully refined topological string on $K_S$, which is anticipated to be a $\CC^\star_{\epsilon_1} \times \CC^\star_{\epsilon_2}$-equivariant theory of the Calabi--Yau $5$-fold $K_S \times \CC^2$. In this setting, the action on $\CC^2$ has weights $\epsilon_1$ and $\epsilon_2$, while the action
on the fibers of $K_S$ has weight $-\epsilon_1 - \epsilon_2$. For initial developments using sheaf counting, see \cite{NO}, and for a Gromov--Witten-based formulation, see \cite{brini2024refined}. Mathematically, the fully refined topological string should encode the refined BPS invariants of $K_S$ in all genera.

In the limit $\epsilon_1 \to 0$, known in the physics literature as the Nekrasov-Shatashvili limit, only the genus $0$ refined BPS invariants of $K_S$ should contribute, with the K-theoretic definition using the remaining action of $\CC^\star_{\epsilon_2}$ on the fibers of $K_S$ with weight $\epsilon_2$. 
Meanwhile, in this limit, the fixed locus of $\CC^\star_{\epsilon_2}$ acting on 
$K_S \times \CC^2$ is $S \times \mathbb{C}$, with normal weights $-\epsilon_2$ and 
$\epsilon_2$. An expected property of the refined topological string is that, in such cases, it should recover the standard (unrefined) Gromov--Witten theory of 
$S \times \CC$, where $\epsilon_2$ plays the role of the genus expansion parameter \cite[\S 2.3]{NO}. These two different interpretations of the limit $\epsilon_1 \rightarrow 0$ of the refined topological string theory, either using the 3-fold $K_S \subset K_S \times \CC^2$ or the 3-fold $S \times \CC \subset K_S \times \CC^2$, lead to a relationship between enumerative invariants of the form predicted by Conjecture \ref{conj_DT}.

Finally, we note that a version of Conjecture \ref{conj_DT} without point insertions is proved for $S=\PP^2$ and conjectured for del Pezzo surfaces in \cite[Theorem 1.10]{bousseau_tak}.

\subsection{K3 surfaces and refined DT invariants}
Let $S$ be a smooth projective K3 surface. In this case, the natural analogue of Conjecture \ref{conj_DT}, formulated using reduced invariants, is known to hold. 
Indeed, for every algebraic $\beta \in H_2(S,\ZZ)$ with $\beta^2=2h-2$, the Laurent polynomials $BPS_\beta^S(q)$ and $\Omega_\beta^{K_S}(q)$ both coincide with the coefficient of $u^h$ in the power series expansion of 
\begin{equation}\label{eq_kkv}
\prod_{n \geq 1} \frac{1}{(1-u^n)^{20}(1-qu^n)^2(1-q^{-1}u^n)^2}\,.\end{equation}
For $BPS_\beta^S(q)$, this follows from the KKV conjecture \cite{KKV}, proved in \cite{KKV_proof}, while the corresponding result for $\Omega_\beta^{K_S}(q)$ is proved in \cite[Remark 4.7]{thomas2024refined}. 
We also refer to \cite[\S 7, Theorems 25-28]{bousseau2017tropical} 
for additional results involving point insertions for K3 and abelian surfaces  which can be viewed as special cases of Conjecture \ref{conj_DT}. 
While these results for K3 and abelian surfaces may initially appear coincidental, Conjecture \ref{conj_DT} suggests that similar results should hold in a much broader context. 
The only coincidence is that $K_S=S\times \CC$ for K3 and abelian surfaces, whereas in general these two 3-folds play distinct roles.

\subsection{Real K3 surfaces and Welschinger invariants}
\label{sec_K3_real}
Conjecture \ref{conj_bps_w} predicts that the specialization of the BPS polynomials
at $q=-1$ for rational surfaces coincides with Welschinger invariants counting real rational curves. This naturally leads to the broader question of whether the specialization of BPS polynomials at $q=-1$ still has a meaningful interpretation within 
real algebraic geometry for more general surfaces. In general, the answer is unclear due to the lack of generality of the definition of Welschinger invariants. However, we remark in this section that the answer is positive when $S$ is a smooth projective K3 surface.

Indeed, consider a smooth projective real K3 surface, and an algebraic primitive class $\beta \in H_2(S,\ZZ)$ with $\beta^2=2h-2$.
Then, there exist only finitely many rational curves in $S$ of class $\beta$, all with $h$ nodes.
The count $W_\beta^S$ of the corresponding real curves with Welschinger sign is determined by
\cite[Corollary 0.1 (2)]{KR_realK3} as the coefficient of $u^h$ in the power series expansion of
\begin{equation}\label{eq_K3}
\prod_{r\geq 1} \frac{1}{(1+u^r)^{e_\RR}} 
\prod_{s \geq 1} 
\frac{1}{(1-u^{2s})^{\frac{e_\CC -e_\RR}{2}}}=\prod_{n \geq 1} \frac{1}{(1+u^n)^{\frac{e_\CC+e_\RR}{2}}(1-u^n)^{\frac{e_\CC-e_\RR}{2}}}\,,\end{equation}
where $e_\CC=24$ is the topological Euler characteristic of $S$, and $e_\RR$ is the topological Euler characteristic of the real locus of $S$.
The formula \eqref{eq_K3} agrees with the specialization at $q=-1$ of \eqref{eq_kkv} if and only if $e_\RR=-16$.
Therefore, we have $BPS_\beta^S(-1)=W_\beta^S$ for real K3 surfaces with $e_\RR=-16$.
By \cite[p85, Figure 3]{real_enriques}, real K3 surfaces with $e_\RR=-16$ exist and have real loci given by either a genus $9$ surface, or the union of a sphere and a genus $10$ surface.  
The latter case is distinguished in several ways. For example, the corresponding real K3 surfaces are maximal, meaning that they saturate the Smith-Thom inequality 
$\dim H_\star(S(\RR),\ZZ/2) \leq \dim H_\star(S,\ZZ/2\ZZ)$. 
Moreover, these real K3 surfaces can be explicitly constructed from toric degenerations with trivial choices of real gluing data by 
\cite[Proposition 9.1]{hulya_real_KN}.

Finally, note that $-16$ is also the signature $\sigma$ of a K3 surface, and so the condition $e_\RR=-16$ can be rewritten as $e_\RR=\sigma$. All the rational surfaces $S_n$ with their standard real structure appearing in Conjecture \ref{conj_bps_w} also satisfy the condition $e_\RR=1-n=\sigma$.
This raises the question of whether there exists a more general real interpretation of $BPS_\beta^S(-1)$ applicable to all real surfaces 
with $e_\RR=\sigma$.

\bibliographystyle{plain}
\bibliography{bibliography}
%-------------------------------------------------------------------------------

\end{document}